\documentclass[a4paper, 11pt]{article}
\usepackage{amsmath,amsfonts,amssymb, amsthm,mathrsfs}
\usepackage[top=35truemm,bottom=35truemm,left=25truemm,right=25truemm]{geometry}

\usepackage[pdftex]{graphicx}
\usepackage{here}
\usepackage{comment}
\usepackage{authblk}
\usepackage{indentfirst}
\usepackage{mathabx}

\newtheorem{prop}{Proposition}[section]
\newtheorem{lemm}[prop]{Lemma}
\newtheorem{coro}[prop]{Corollary}
\newtheorem{theo}[prop]{Theorem}

\theoremstyle{definition}
\newtheorem{defi}[prop]{Definition}
\newtheorem{exam}[prop]{Example}

\title{\bf ON INVARIANTS OF SURFACES IN THE 3-SPHERE}
\author{Hiroaki Kurihara}
\date{\empty}

\begin{document}
\maketitle

\begin{abstract}
In this paper we study isotopy classes of closed connected orientable surfaces in the standard $3$-sphere.
Such a surface splits the $3$-sphere into two compact connected submanifolds, and by using their Heegaard splittings, we obtain a $2$-component handlebody-link.
In this paper, we first show that the equivalence class of such a 2-component handlebody-link up to attaching trivial $1$-handles can recover the original surface.
Therefore, we can reduce the study of surfaces in the $3$-sphere to that of $2$-component handlebody-links up to stabilizations.
Then, by using $G$-families of quandles, we construct invariants of $2$-component handlebody-links up to attaching trivial $1$-handles, which lead to invariants of surfaces in the $3$-sphere.
In order to see the effectiveness of our invariants, we will also show that our invariants can distinguish certain explicit surfaces in the $3$-sphere.
\end{abstract}

\section{Introduction}
Throughout this paper, we work in the PL category.
In this paper, we will study closed connected orientable surfaces embedded in the standard 3-sphere  $S^3$.
Such embedded surfaces were extensively studied by Fox, Homma, Tsukui, and Suzuki around 1950s--1970s (e.g., see \cite{fox, hom1, suz2, tsu1}).
In their studies, $3$-manifold theory and fundamental group techniques were mainly used.

A closed connected orientable surface in $S^3$ splits $S^3$ into two compact connected submanifolds of codimension 0 with common boundary.
Then, by considering their Heegaard splittings, we obtain a pair of a handlebody and a compression body from each submanifold.
Moreover, it is known as the Reidemeister--Singer theorem that any two Heegaard splittings of a compact connected orientable 3-manifold become equivalent after finitely many stabilizations.
Hence, given a surface embedded in $S^3$, we can associate a certain 2-component handlebody-link up to stabilizations.
Then, we show that an isotopy class of the original surface is uniquely determined from the equivalence class of an associated $2$-component handlebody-link.
Therefore, we can reduce the study of surfaces in $S^3$ to that of 2-component handlebody-links in a certain sense.

Handlebody-links (-knots)  have recently been studied quite actively (e.g., see \cite{ishii3, ishii1, ishii2}).
For example, by using a $G$-family of quandles, which is an algebraic system with binary operations parametrized by the elements of a group $G$, we can construct invariants of handlebody-links (see \cite{ishii1}).

One of the main purposes of the present paper is to construct invariants of surfaces embedded in $S^3$ by means of the associated 2-component handlebody-links and $G$-families of quandles.

The present paper is organized as follows.
In Section~2, we first recall the notion of compression bodies.
Handlebodies can be regarded as a special class of compression bodies.
We will also recall the notions of Heegaard splittings of $3$-manifolds and their stabilizations.
Then, we review the Reidemeister--Singer theorem, which will be used in Section~3.
We also introduce the notion of $G$-familes of quandles and known results of them.

In Section~3, we construct invariants of surfaces embedded in $S^3$.
Plenty of invariants of handlebody-links were given in \cite{ishii1}.
We first use $X$-colorings for the construction of surfaces in $S^3$.
We observe a variation of the cardinality of the set of $X$-colorings due to a stabilization.
Then, we give a technique which cancels the variation of the cardinality of the set of $X$-colorings after a stabilization of a $2$-component handlebody-link.
Hence, we obtain an invariant of handlebody-links up to stabilizations.
We consider such an invariant for each connected component of a $2$-component handlebody-link. 
Then, as one of the results,  we obtain an invariant as the unordered pair of rational numbers.
Moreover, we will show that, by using  a similar argument, we can construct another invariant of surfaces in $S^3$ by using another quandle invariant which is defined by using cohomology theory of $G$-families of quandles \cite{ishii1}.

In Section~4, we compute our invariants constructed in Theorem \ref{thm:6} for two explicit examples of surfaces in $S^3$.
By using a $G$-family of quandles with a simple algebraic structure, we show that our invariant distinguish the two surfaces by a small amount of calculation.

In Section~5, we will show that for two surfaces in $S^3$, if the associated 2-component handlebody-links are equivalent up to attaching 1-handles, then two surfaces are isotopic. 
We also study surface in $S^3$ from geometric viewpoints.
We construct a geometric invariant of surfaces, which is an analogy of the unknotting tunnel number of a knot.
Then, we study the relationship between a surface in $S^3$ and the closures of the connected components of the exterior of the surface.

Throughout the paper, for a manifold $M$, we denote by $\partial M$ and ${\rm int}(M)$ the boundary and the interior of $M$, respectively.
For a subset $N$ of $M$, denote by ${\rm cl}(N)$ the closure of $N$ in $M$.
Furthermore, for two sets $M_1$ and $M_2$, we denote by $M_{1}\sqcup M_{2}$ the disjoint union of $M_{1}$ and $M_{2}$.
For a set $Z$ we denote by $\# Z$ the cardinality of $Z$.
We also denote by ${\rm id}_{Z}$ the identity map of $Z$.
For a group $G$, we denote by $e_{G}$ the identity element of $G$.
We denote by $g(F)$ the genus of a closed connected orientable surface $F$.
For a compact connected orientable surface $F$ in a compact connected orientable $3$-manifold $M$, we denote by $N(F)$ a regular neighborhood of $F$ in $M$.
A multiset is a set whose elements allow duplications.
A duplication of an element of a multiset is called the multiplicity of the element.

\section{Preliminaries}
In this section, we present definitions and known results that will be used when we construct invariants of embedded surfaces in $S^3$.
The main theme of this paper is to study closed connected orientable surfaces in $S^3$ as we mentioned.

Let $F$ be a closed connected orientable surface embedded in $S^3$.
We denote by $V_{F}$ and $W_{F}$ the closures of the connected components of $S^{3}\setminus F$.
Then, by considering Heegaard splittings of $V_F$ and $W_F$, we will obtain a 2-component handlebody-link from the original surface $F$. 
Moreover, we study the relationship between a surface in $S^3$ and the associated 2-component handlebody-link.
We also recall definitions and related results of $G$-family of quandles.
By the Alexander theorem (\cite{ale}), all $2$-spheres in $S^3$ are isotopic, then we consider surfaces of genus greater than $0$.

\subsection{$3$-manifolds and Heegaard splittings}
In this section, we recall terminologies of $3$-manifolds and Heegaard splittings.
We also introduce several known results related to Heegaard splittings.
Originally, Heegaard splittings were introduced for closed 3-manifolds.
After that, the definition of a Heegaard splitting was extended to compact 3-manifolds, and such a splitting is called a generalized Heegaard splitting (Definition \ref{heegaard}).
We start from the definition of compression bodies.

\begin{defi}
[\cite{{SJS1}, {ScT1}}]
Let $\Sigma$ be a closed connected orientable surface.
Consider the product manifold $\Sigma\times[0,1]$.
Then, attach $2$-handles along mutually disjoint simple closed curves on $\Sigma\times\{0\}$, and cap off every resulting $2$-sphere component by a $3$-handle.
The resulting $3$-manifold $C$ is called a \emph{compression body}.
We denote $\Sigma\times\{1\}$ by $\partial_{+}C$ and $\partial C\setminus \partial_{+}C$ by $\partial_{-}C$.
If $C$ is constructed without any $3$-handle, then $C$ can be obtained from $\partial_{-}C\times [0,1]$ by attaching mutually disjoint $1$-handles on $\partial_{-}C\times\{1\}$.
A compression body $C$ is said to be a {\it handlebody} if $\partial_{-}C=\emptyset$.
We define the {\it genus} of a handlebody as the genus of its boundary.
\end{defi}

Let us recall  handlebody-knots (-links) and their related terminologies.
A {\it handlebody-knot} of genus $g$ is a handlebody of genus $g$ embedded in  $S^3$.
Two handlebody-knots are {\it equivalent} if there exists an ambient isotopy of $S^3$ which maps one to the other. 
Let $H_{1}, H_2, \ldots, H_{n}$ be mutually disjoint handlebody-knots.
Then, $L:=H_{1}\sqcup H_{2}\sqcup\cdots\sqcup H_{n}$ is called a {\it handlebody-link}, or sometimes an \emph{$n$-component handlebody-link}.
Two handlebody-links are ${\it equivalent}$ if there exists an ambient isotopy of $S^3$ which maps one to the other.
A handlebody-knot $H$ (-link $L$) is represented by a spatial trivalent graph $K$ if a regular neighborhood of $K$ is ambient isotopic to $H$ ($L$). 
Here, a spatial trivalent graph is a finite graph embedded in $S^3$ or $\mathbb{R}^3$ such that each vertex is of valence three.
It is known that any handlebody-knot (-link) can be represented by a spatial trivalent graph 
(see \cite{ishii3}).
A {\it diagram} $D$ of a handlebody-knot $H$ ($L$) is a diagram of a spatial trivalent graph $K$ of $H$ ($L$) obtained by projecting $K$ to the $2$-sphere $S^2$ or the $2$-plane $\mathbb{R}^{2}$.
A diagram $D$ of a handlebody-knot (-link) consists of arcs, which are parts of curves, and each of the endpoints of an arc are a vertex or an undercrossing.
At each crossing of $D$, an arc is called an \emph{under-arc} if one of the its endpoints is undercrossing.
Otherwise, an arc is called an \emph{over-arc}. 
An oriented diagram is a diagram of whose each of an arc is oriented. 

Let $M$ and $N$ be a compact 3-manifold and a submanifold of $M$, respectively.
Then, $N$ is said to be {\it properly embedded} in $M$ if $\partial N\subset \partial M$ and $\text{int}(N)\subset \text{int}(M)$.

A properly embedded $2$-disk $D$ in $M$ is said to be {\it inessential} if there exists a 2-disk $D^{'}$ in $\partial M$ such that $\partial D = \partial D^{'}$ and $D\cup D^{'}$ is the boundary of a 3-ball in $M$.
A properly embedded $2$-disk $D$ is said to be {\it essential} if $D$ is not inessential.
Moreover, let $M$ be a compact connected $3$-manifold. 
Then, an essential $2$-disk $D$ in $M$ is said to be {\it separating} if $M\setminus D$ consists of two connected components.
Otherwise, $D$ is said to be {\it non-separating}.

\begin{defi}
Let $M$ be a compact connected orientable $3$-manifold.
Then, $M$ is said to be {\it irreducible} if any $2$-sphere $S$ in $M$ bounds a $3$-ball in $M$.
\end{defi}
We refer the reader to \cite{mac, hem} about the $3$-manifold theory used in this paper.

\begin{defi}[\cite{wal}]
Let $M$ be a compact orientable $3$-manifold.
Let $F$ be a compact orientable surface in ${\rm int}(M)$ or in $\partial M$.
Then, $F$ is said to be {\it compressible} if $F$ satisfies either of the following conditions.
\begin{itemize}
\item[(i)] There exists a properly embedded $2$-disk $D$ in $M$ such that $D\cap {\rm int}(F)=\partial D$ and $\partial D$ is non-contractible in ${\rm int}(F)$
.\item[(ii)] There exists a $3$-ball $B$ such that $B\cap F=\partial B$.
\end{itemize}
The surface $F$ is said to be {\it incompressible} if $F$ is not compressible.
\end{defi}

The following theorem and lemma given in \cite{mac} and \cite{wal}, respectively, will be used  to see a relationship between the set of equivalence classes of 2-component handlebody-links up to attaching $1$-handles and the set of isotopy classes of embedded surfaces.
We begin by introducing the notion of incompressible neighborhood.

\begin{defi}[Incompressible neighborhood, \cite{mac}]\label{incomp}
Let $M$ be an orientable irreducible $3$-manifold.
Let $F$ be a compressible boundary component of $M$.
Then, a $3$-dimensional submanifold $V$ of $M$ is said to be an {\it incompressible neighborhood} of $F$ if $V$ satisfies the following conditions.
\begin{itemize}
\item[(i)]The $3$-manifold $V$ is a compact connected submanifold of $M$ such that $F\subset V\subset M$ and $\partial V\setminus F\subset {\rm int}(M)$.
\item[(ii)]The $2$-manifold $\partial V\setminus F$ is incompressible in $M$.
\item[(iii)]For some $x_{0}\in F$, 
$$
{\rm Image}(\pi_{1}(V,x_{0})\to\pi_{1}(M,x_{0}))={\rm Image}(\pi_{1}(F,x_{0})\to\pi_{1}(M,x_{0})).
$$
\end{itemize}
\end{defi}

\begin{theo}[{\rm \cite{mac}}\label{mac}]
Let $M$ be an orientable irreducible $3$-manifold.
Let $F$ be a compressible boundary component of $M$.
Then, there exists an incompressible neighborhood $V$ of $F$, which is unique up to isotopy of $M$.
\end{theo}

For the definition of an incompressible neighborhood of a compressible boundary component, the reader is referred to Definition \ref{incomp}.

\begin{lemm}[{\rm\cite{wal}}\label{irreducible}]
Let $M$ be a compact connected orientable $3$-manifold.
Let $F$ be the union of mutually disjoint incompressible surfaces in $M$.
We set $M^{'}={\rm cl}(M\setminus N(F))$.
Then $M^{'}$ is irreducible if and only if $M$ is irreducible.
\end{lemm}

We introduce the notion of Heegaard splittings of 3-manifolds.
Originally, Heegaard splittings were introduced to represent a closed connected orientable 3-manifold by the union of two handlebodies along their common boundaries.
In the context of knot theory, the Heegaard genus of a Heegaard splitting of the exterior $E(K)$ of a knot $K$ is related to the tunnel number of $K$.

\begin{defi}[Heegaard Splittings\label{heegaard}]
Let $M$ be a compact connected orientable 3-manifold possibly with boundary.
Fix a partition of $\partial M$ as $\partial M =\partial_{1}M\sqcup\partial_{2}M$.
A {\it Heegaard splitting} of $M$ is a decomposition of $M$ into two compression bodies $C_1$ and $C_2$ such that $M=C_{1}\cup C_{2}$, $\partial_{+}C_{1}=S=\partial_{+}C_{2}$, $C_{1}\cap C_{2} =S$, $\partial_{-}C_{1}=\partial_{1}M$, and $\partial_{-}C_{2}=\partial_{2}M$.
We call $\partial_{+}C_{1}=S=\partial_{+}C_{2}$ a {\it Heegaard surface} of $M$.
We denote by $(M, S)$ a Heegaard splitting of $M$ with a Heegaard surface $S$.
We also denote by $C_{1}\cup_{S}C_{2}$ a Heegaard splitting of $M$ consisting of two compression bodies $C_{1}$ and $C_2$ with a Heegaard surface $S$.
The minimal genus of Heegaard surfaces of $M$ is called the {\it Heegaard genus} of $M$.
In particular, a Heegaard splitting of $M$ with the minimal genus is called a \emph{minimal genus Heegaard splitting} of $M$.
\end{defi}

Concerning Heegaard splittings of 3-manifolds, it is known as Moise's theorem that
every compact connected orientable $3$-manifold possibly with boundary admits a Heegaard splitting
 {\rm (\cite{{moi}})}.

We define a parallel arc in a compression body.
Parallel arcs will be used when we define the stabilization of a Heegaard splitting of a $3$-manifold.
\begin{defi}[Parallel arc]
Let $C$ be a compression body.
A properly embedded arc $\alpha$ in $C$ with $\partial C\cap \partial\alpha=\partial_{+}C\cap\partial\alpha$ is {\it parallel} to an arc $\beta$ in $\partial_{+}C$ with $\partial\beta=\partial\alpha$ if there is an embedded disk $D$ in $C$ such that $\partial D = \alpha \cup \beta$.
\end{defi}

\begin{defi}[Stabilization\label{def:sbz}]
Let $M$ be a compact connected orientable 3-manifold possibly with boundary.
Let $(M, S)$ be a Heegaard splitting of $M$ with a Heegaard surface $S$. 
The following procedure to construct a new Heegaard splitting $(M, S^{'})$ from $(M, S)$ is called a {\it stabilization}.
Suppose that $M=C_{1}\cup_{S}C_{2}$ is a Heegaard splitting of $M$ consisting of two compression bodies $C_1$ and $C_{2}$ with $\partial_{+}C_{1}=S=\partial_{+}C_{2}$.
Take a parallel arc $\alpha$ in $C_{2}$ to an arc $\beta$ in $\partial_{+}C_{2}$.
Then, we remove a tubular neighborhood $N(\alpha)$ of $\alpha$ from $C_2$, take the closure of $C_{2}\setminus N(\alpha)$, and add $N(\alpha)$ to $C_1$, namely, we have $C_{1}^{'}:=C_{1}\cup N(\alpha)$ and $C_{2}^{'}:=\text{cl}(C_{2}\setminus N(\alpha))$. 
We can show that $C_{1}^{'}$ and $C_{2}^{'}$ are also compression bodies satisfying $C_{1}^{'}\cup_{S^{'}}C_{2}^{'}=M$, where $S^{'}:=\partial C^{'}_{1}=\partial C^{'}_{2}$.
\end{defi}

Hence we obtained a new Heegaard splitting $(M,S^{'})$ of $M$ from the given Heegaard splitting $(M, S)$.
We also have $g(S^{'})=g(S)+1$.

We give an example of a Heegaard splitting of a handlebody of genus $g$.
\begin{exam}[The trivial Heegaard splitting]
Let $H_{g}$ be a handlebody of genus $g$.
Set $\partial H_{g}=F$.
Consider a parallel surface $F^{'}$ to $F$ inside of $H_g$.
Then, the closures of the connected components of the exterior of $F^{'}$ consists of a handlebody of genus $g$ and a compression body with the same boundaries.
These handlebody and compression body give a Heegaard splitting of $H_g$.
Such a splitting is called the {\it trivial splitting} of $H_g$.
\end{exam}

Concerning Heegaard splittings of the handlebody of any genus $g$, it is known that  
all Heegaard splittings of a handlebody of genus $g$ are standard, that is, they are obtained from the trivial Heegaard splitting by applying a finite number of stabilizations {\rm (\cite{ScT1})}.

We introduce the notion of the equivalence of two Heegaard splittings of a compact connected orientable 3-manifold $M$ possibly with boundary.
Let ($M,S$) and ($M,S^{'}$) be Heegaard splittings of $M$ with Heegaard surfaces $S$ and $S^{'}$.
They are {\it equivalent} if there exists an ambient isotopy of $M$ which maps $S$ to $S^{'}$.
The following theorem, known as the Reidemeister--Singer theorem, plays an important role when we construct invariants of embedded surfaces in $S^3$.

\begin{theo}[\label{thm:1}{\rm Reidemeister--Singer theorem \cite{{rei}, {sin}}}]
Let $M$ be a compact connected orientable $3$-manifold possibly with boundary.
We fix a partition of $\partial M$ as $\partial M=\partial_{1}M\sqcup\partial_{2}M$. 
Then, any two Heegaard splittings of $M$ with the fixed partition of $\partial M$ become equivalent after finitely many  stabilizations.
\end{theo}

The following is also known as Waldhausen's theorem.
\begin{theo}[\label{walthm}{\rm \cite{wal1}}]
A Heegaard splitting of $S^3$ is unique up to isotopy in every genus of a Heegaard surface.
\end{theo}

In Definition \ref{def:sbz}, we introduced the notion of the stabilization of a Heegaard splitting of a compact conneccted orientable $3$-manifold.
Then, let us define a stabilization of handlebody-links.
This definition is induced from the definition of the stabilization of a Heegaard splitting of a compact connected orientable $3$-manifold with connected boundary.
Two handlebody-links $L_1$ and $L_2$ are said to be \emph{separated} if there exists $3$-balls $B^{3}_{1}$ and $B^{3}_{2}$ in $S^3$ such that $B^{3}_{1}\cap B^{3}_{2}=\partial B^{3}_{1}\cap \partial B^{3}_{2}$, $L_{1}\subset {\rm int}(B^{3}_{1})$, and $L_{2}\subset {\rm int}(B^{3}_{2})$.
Similarly, $n$ handlebody-links $L_{1}, L_{2}, \ldots, L_{n}$ are said to be \emph{separated} if there exists $3$-balls $B^{3}_{1}, B^{3}_{2}, \ldots, B^{3}_{n}$ in $S^3$ such that $B^{3}_{i}\cap B^{3}_{j}=\partial B^{3}_{i}\cap \partial B^{3}_{j}$ (whenever $i\neq j$) and $L_{1}\subset {\rm int}(B^{3}_{1}), L_{2}\subset {\rm int}(B^{3}_{2}), \ldots, L_{n}\subset {\rm int}(B^{3}_{n})$.



\begin{defi}[Stabilization of handlebody-links]\label{hls}
Let $H$ be a handlebody-knot.
Then, the {\it stabilization} of $H$ is the disk sum of $H$ and the standard solid torus $T$, denoted by $H\natural T$, by a $2$-disk $D$ on $\partial H$ so that the connected components of $H\natural T\setminus N(D)$ are separated by a $3$-ball.
Let $L=H_{1}\sqcup H_{2}$ be a 2-component handlebody-link.
Let $T^{1}$ and $T^{2}$ be the standard solid tori.
Then, the {\it stabilization} of $L$ is stabilizations of $H_{1}$ with $T^{1}$ by a $2$-disk $D_1$ on $\partial H_1$ or that of $H_{2}$ with $T^{2}$ by a $2$-disk $D_2$ on $\partial H_2$ so that the connected components of ${\rm cl}(L\natural T^{1}\natural T^{2} \setminus \sqcup_{i=1}^{2}N(D_{i}))$ are separated by $3$-balls in $S^3$.
Two handlebody-links are \emph{stably equivalent} if they are equivalent after finitely many stabilizations.

\end{defi}

We note that if both $H_1$ and $H_2$ are stabilized, then the attached solid tori are separated by some $3$-balls in $S^3$.

\subsection{$G$-family of quandles and handlebody-links}

Let us now go on to introducing the notions of quandles and $G$-families of quandles. 
A lot of invariants of links have been obtained by using quandles.
Moreover, $G$-families of quandles can be used for studying handlebody-links and gives plenty of invariants(refer to \cite{ishii1}).
We start by the definition of  quandles.
We refer the reader \cite{ishii1} about the theory of  $G$-families of quandles used in the paper.

\begin{defi}
Let $X$ be a non-empty set with a binary operation $\ast:X \times X\rightarrow X$. The pair $(X,\ast)$ is a {\it quandle} if $\ast$ satisfies the following conditions.
\begin{itemize}
\item [(i)] $x\ast x = x$ for any $x\in X$.
\item [(ii)] The map $S_{x}:X\rightarrow X$ defined by $S_{x}(y) = y\ast x$ is bijective for any $x\in X$.
\item [(iii)] $(x\ast y) \ast z = (x \ast z) \ast (y \ast z)$ for any $x, y, z \in X$ .
\end{itemize}
\end{defi}

\begin{defi}
Let $G$ be a group and $X$ be a non-empty set with a family of binary operations $\ast_{g}:X\times X\rightarrow X$ parametrized by $g \in G$, respectively. The pair $(X, \{\ast_{g}\}_{g\in G})$ is a {\it $G$-family of quandles} if for any $x,y,z\in X$ and any $g,h\in G$,
 $\ast_{g}$ satisfies the following conditions (see \cite{ishii1}).
\begin{itemize}
\item[(i)] $x \ast_{g} x = x$.
\item[(ii)]$x \ast_{gh} y = (x \ast_{g} y ) \ast_{h} y$ and $x \ast_{e_{G}} y = x$.
\item[(iii)]$(x\ast_{g}y)\ast_{h}z = (x\ast_{h}z)\ast_{h^{-1}gh}(y\ast_{h}z)$.
\end{itemize}
 \end{defi}

We note that for a $G$-family of quandles, the pair ($X,\ast_{g}$) is a quandle for each $g\in G$.

\begin{defi}
Let $G$ and $(X,\{\ast_{g}\}_{g\in G})$ be a group and a $G$-family of quandles, respectively. 
Set $Q=X\times G$.
We define the binary operation $\ast:Q\times Q\rightarrow Q$ by $(x,g)\ast(y,h)=(x\ast_{h}y,h^{-1}gh)$. Then, the pair $(Q,\ast)$ is a quandle called the {\it associated quandle} of $X$.
\end{defi}

Using a $G$-family of quandles $(X, \{\ast_{g}\}_{g\in G})$, we can introduce the notion of an $X$-coloring for an oriented diagram $D$ of a habdlebody-link. 
We denote by $\mathscr{A}(D)$ the set of arcs of $D$.
The normal orientation of an oriented arc is given by rotating an orientation of the arc counterclockwise by $\pi/2$.
The normal orientation of an oriented arc is represented by an arrow on the arc (see Figure \ref{fig:13}).
Using a $G$-family of quandles $(X, \{\ast_{g}\}_{g\in G})$, we can introduce the notion of an $X$-coloring for an oriented diagram $D$ of a habdlebody-link.

\begin{defi}[$X$-colorings]
Let $G$ be a group and ($X,\{\ast_{g}\}_{g\in G}$) a $G$-family of quandles, respectively.
Let $D$ be an oriented diagram of a handlebody-link.
A map $C:\mathscr{A}(D)\rightarrow Q=X\times G$ is an {\it $X$-coloring} of $D$ if $C$ satisfies the following conditions:
\begin{itemize}
\item [(i)] at each crossing $\chi$ of $D$, the map $C$ satisfies $C(\chi_{2})=C(\chi_{1}) \ast C(\chi_{3})$, and
\item [(ii)] at each vertex $\omega$ of $D$ and for the two natural projections $p_{X}:Q\rightarrow X$ and $p_{G}:Q\rightarrow G$, the map $C$ satisfies
\begin{eqnarray*}
\left\{
\begin{array}{l}
p_{X}\circ C(\alpha_{1})=p_{X}\circ C(\alpha_{2})=p_{X}\circ C(\alpha_{3}),\\
(p_{G}\circ C(\alpha_{1}))^{\varepsilon(\omega,\alpha_{1})}\cdot (p_{G}\circ C(\alpha_{2}))^{\varepsilon(\omega,\alpha_{2})}\cdot (p_{G}\circ C(\alpha_{3}))^{\varepsilon(\omega,\alpha_{3})}=e_{G}, 
\end{array}
\right.
\end{eqnarray*}
where $\chi_{1}$, $\chi_{2}$ are under arcs and $\chi_{3}$ is an over-arc of $D$ (Figure \ref{fig:13}).
Furthermore, $\varepsilon(\omega,\alpha_{i})$ is the sign of an arc $\alpha_{i}$ at $\omega$ which is defined as follows (see Figure \ref{fig:14}) :
\begin{eqnarray*}
\varepsilon(\omega,\alpha_{i})=
\left\{
\begin{array}{l}
1,\quad\text{if the orientation of $\alpha_{i}$ points into $\omega$,}\\
-1,\quad\text{otherwise}.
\end{array}
\right.
\end{eqnarray*}
\end{itemize}
We denote by $\text{Col}_{X}(D)$ the set of $X$-colorings of $D$.
\end{defi}
\begin{figure}[H]
\begin{center}
\includegraphics[width=12cm, height=8cm]{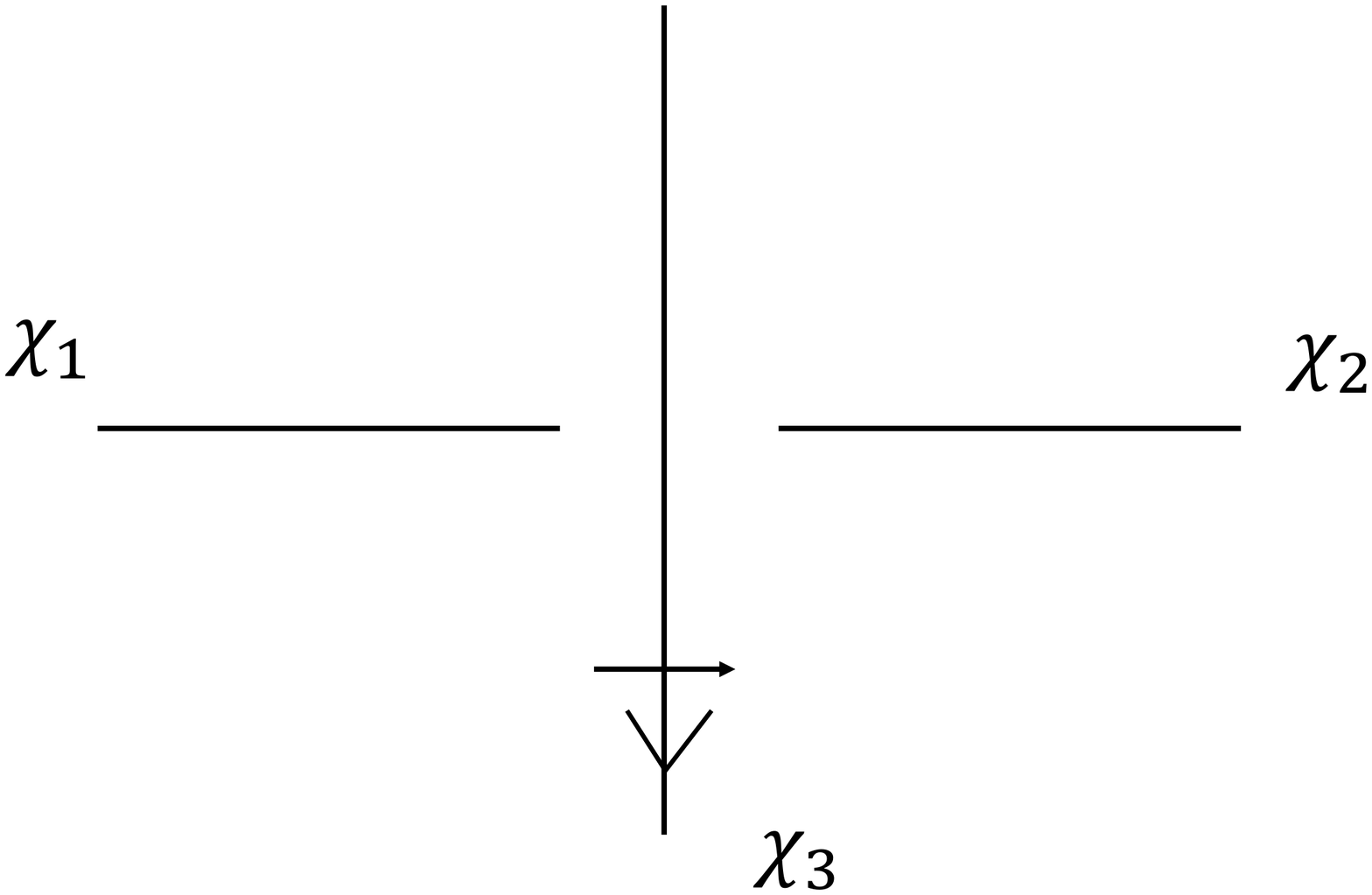}
\end{center}
\vspace{-13mm}
\caption{Crossing $\chi$ and the normal orientation of the over-arc $\chi_{3}$} 
\label{fig:13}
\end{figure}

\begin{figure}[H]
\begin{center}
\includegraphics[width=12cm, height=8cm]{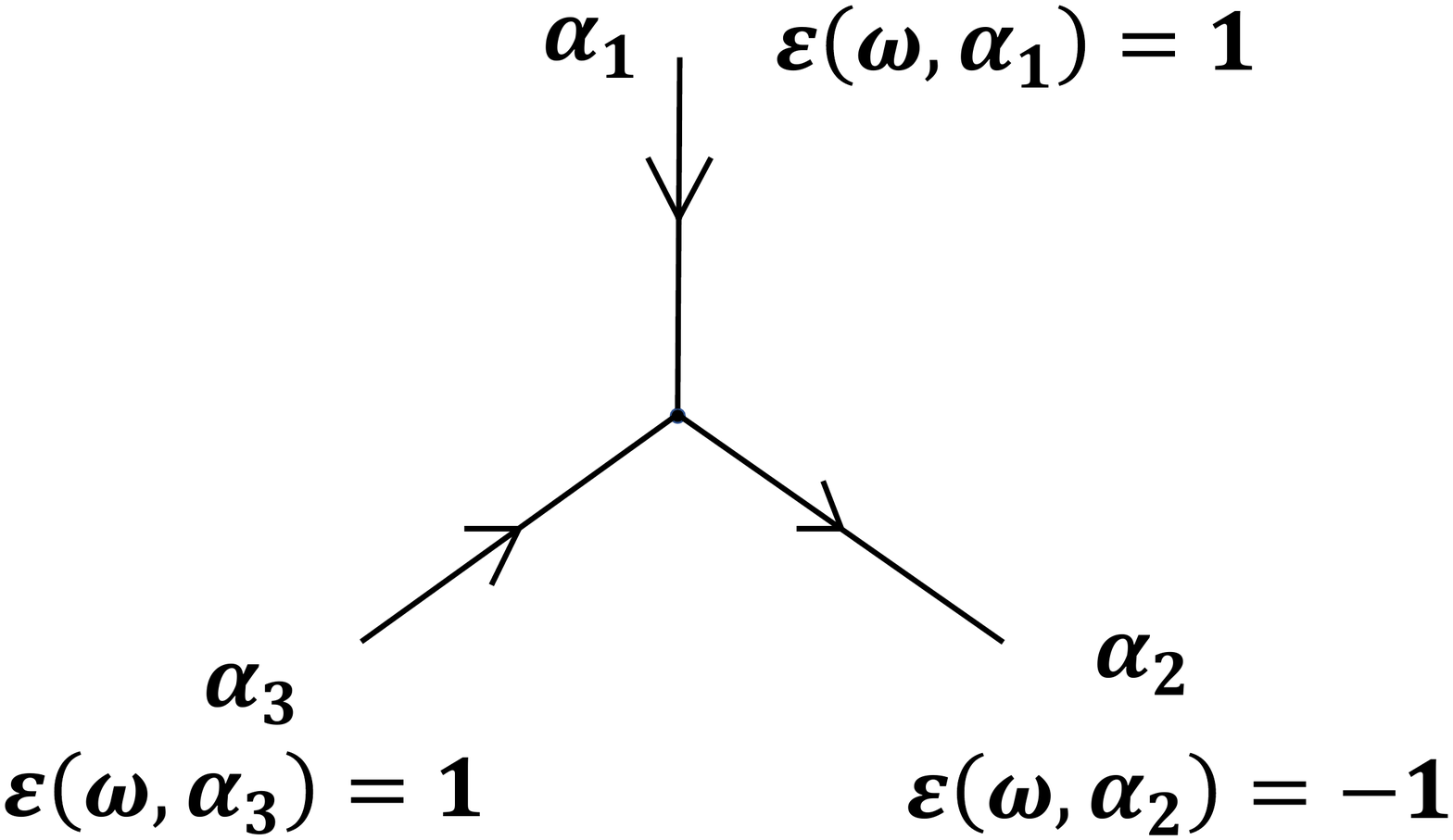}
\end{center}
\vspace{-13mm}
\caption{Vertex $\omega$}
\label{fig:14}
\end{figure}

\begin{theo}\label{thm:4}{\rm (\cite{ishii1})}
Let $G$ be a group and $(X,\{\ast_{g}\}_{g\in G})$ be a $G$-family of quandles, respectively.
Let $D$ be a diagram of an oriented spatial trivalent graph of a handlebody-link.
Then, the cardinality $ \#{\rm Col}_{X}(D)$ is an invariant of the handlebody-link.
\end{theo}

Let $D$ be an oriented diagram of a handlebody-knot $H$.
Let $H^{'}$ be a handlebody-knot obtained from the disk sum of $H$ and the standard solid torus.
Then, a diagram $D^{'}$ of $H^{'}$ is obtained from $D$ by attaching an edge and a circle component to an arc of $D$.
Therefore, an orientation of $D^{'}$ is induced from that of $D$ except for the edge and the circle component.

The following lemma will be used in Section~3 for constructing invariants of embedded surfaces.
\begin{lemm}\label{lem:2}
Let $G$ and $X$ be a finite group and a finite set, respectively, and $(X,\{\ast_{g}\}_{g\in G})$ be a $G$-family of quandles.
Let $F$ be a closed connected orientable surface in $S^3$.
We denote by $V_{F}$ and $W_{F}$ the closures of the connected components of $S^{3}\setminus F$.
Let $(V_{F}, F_{V})$ be a Heegaard splitting of $V_{F}$.
Let $(V_{F}, F_{V}^{'})$ be a Heegaard splitting obtained from $(V_{F}, F_{V})$ by applying a stabilization.
Let $H_{F_{V}}$ and $H_{F_{V}^{'}}$ be handlebody-knots obtained from the Heegaard splittings $(V_{F}, F_{V})$ and $(V_{F}, F_{V}^{'})$, respectively.
Let $D_{F_{V}}$ be an oriented diagram of $H_{F_{V}}$ and $D_{F_{V}^{'}}$ be a diagram of $H_{F_{V}^{'}}$ with an orientation induced from that of  $D_{F_{V}}$.
Then, we have $\#{\rm Col}_{X}(D_{F_{V}^{'}})=\#{\rm Col}_{X}(D_{F_{V}})\cdot \# G$.
\end{lemm}

\begin{proof}
By the definition of a stabilization and by using an isotopy $S^3$, we may assume that the diagram $D_{F_{V}^{'}}$ is obtained from $D_{F_{V}}$ by attaching an edge $e_{0}$ and an $S^{1}$ component $\beta$ to an arc $\alpha_{0}$ of $D_{F_{V}}$ so that  $e_0$ and $\beta$ do not admit crossings.
Then, we give arbitrary orientations to $e_{0}$ and $\beta$.
For the other arcs, orientations are induced from that of $D_{F_{V}}$ as shown in Figure \ref{stabilization}.
Let $C$ be an $X$-coloring of $D_{F_{V}}$.
Suppose that $C(\alpha_{0})=(x, g)\in X\times G$.
Then, $C$ is extended to an $X$-coloring of $D_{F^{'}_{V}}$ by defining $C(e_{0})=(x, e_{G})$ and $C(\beta)=(x, h)$.
Hence, we obtain $\#\text{Col}_{X}(D_{F_{1}^{'}})\geq \# \text{Col}_{X}(D_{F})\cdot \# G$.
Conversely, for any $C(\beta)=(x, h)$ we have $C(e_{0})=(x, e_{G})$ by the axiom of an $X$-coloring.
Then we have $\# \text{Col}_{X}(D_{F_{V}^{'}})=\# \text{Col}_{X}(D_{F_{V}})\cdot \# G$.
The same holds for those of $W_{F}$.
\end{proof}

\begin{figure}[H]
\vspace{-5mm}
\begin{center}
\includegraphics[width=12cm, height=8cm]{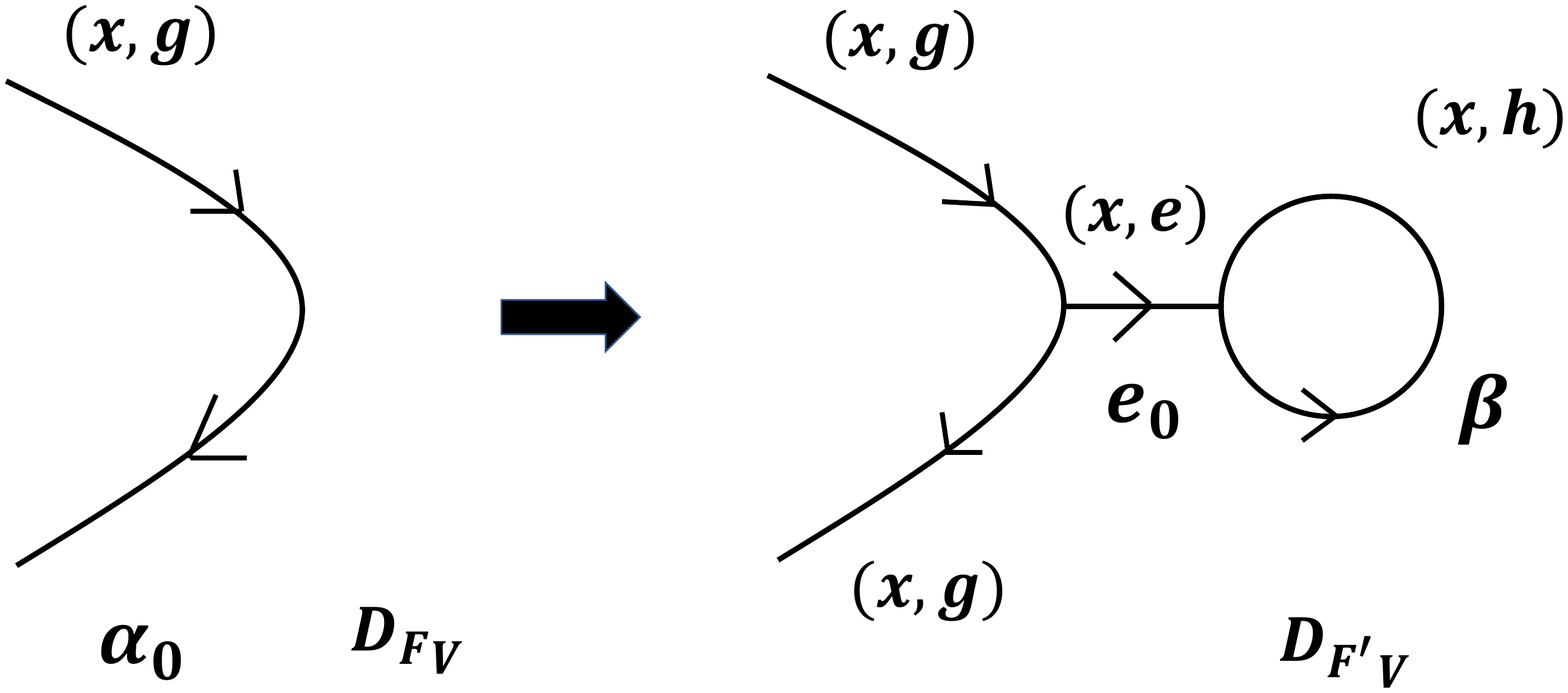}
\vspace{-10mm}
\caption{Attaching an edge $e_{0}$ and a circle component $\beta$ to $D_{F_{V}}$}
\label{stabilization}
\end{center}
\end{figure}

Let us introduce the notions of an $X$-set $Y$ and an $X_{Y}$-coloring of a diagram of a handlebody-link (refer to \cite{ishii1}).
\begin{defi}
Let $(X, \{\ast_{g}\}_{g\in G})$ be a $G$-family of quandles, and let $Y$ be a non-empty set with a family of maps $\bar{\ast}_{g}:Y\times X\to Y$ parametrized by $g\in G$.
The pair $(Y, \{\bar{\ast}_{g}\}_{g\in G})$ is called an {\it $X$-set} if for any $y\in Y$, $x, x^{'}\in X$, and any $g, h\in G$, the following conditions are satisfied.
\begin{itemize}
\item[(i)]$y \mathbin{\bar{\ast}_{gh}} x = (y \mathbin{\bar{\ast}_{g}} x ) \mathbin{\bar{\ast}_{h}} x$ and $y \mathbin{\bar{\ast}_{e_{G}}} x = y$.
\item[(ii)]$(y\mathbin{\bar{\ast}_{g}x}) \mathbin{\bar{\ast}_{h}}x^{'} = (y \mathbin{\bar{\ast}_{h}} x^{'})\mathbin{\bar{\ast}_{h^{-1}gh}}(x\mathbin{\ast_{h}}x^{'})$.
\end{itemize}
\end{defi}

Let $D$ be an oriented diagram of a handlebody-link. 
Let $D^{'}$ be a diagram obtained from $D$ by connecting undercrossing arcs at each crossing of $D$.
Then $D^{'}$ admits no crossing.
We call a connected component of $\mathbb{R}^{2}\setminus D^{'}$ a {\it complementary region} of $D$.
We denote by $\mathscr{R}(D)$ the set of complementary regions of $D$. 
We set $y\ast(x, g)=y\mathbin{\bar{\ast}_{g}}x$ for $y\in Y$ and $(x, g)\in X\times G=Q$.
Then, we introduce the notion of {\it $X_{Y}$-colorings}.

\begin{defi}[$X_{Y}$-coloring]
Let $G$ be a group, and $(X, \{\ast_{g}\}_{g\in G})$ and $(Y, \{\bar{\ast}_{g}\}_{g\in G})$ be a $G$-family of quandles and an $X$-set, respectively.
Let $D$ be an oriented diagram of a handlebody-link.
An {\it $X_{Y}$-coloring} of $D$ is a map $C:\mathscr{A}(D)\cup\mathscr{R}(D)\to Q\cup Y$ satisfying the following conditions.
We denote by $\text{Col}_{X}(D)_{Y}$ the set of $X_{Y}$-colorings of the oriented diagram $D$.
\begin{itemize}
\item[C1.]$C(\mathscr{A}(D))\subset Q$ and $C(\mathscr{R}(D))\subset Y$.
\item[C2.]The restriction of $C$ on $\mathscr{A}(D)$ is an $X$-coloring of $D$.
\item[C3.]For an over-arc $\alpha$ and adjacent complementary regions $\alpha_{1}$ and $\alpha_{2}$, $C$ satisfies $C(\alpha_{2})=C(\alpha_{1}) \ast C(\alpha)$, where the normal orientation of $\alpha$ points from $\alpha_1$ to $\alpha_2$ (see Figure \ref{xy-col}).
\end{itemize}
\end{defi}

\begin{figure}[H]
\begin{center}
\includegraphics[width=12cm, height=8cm]{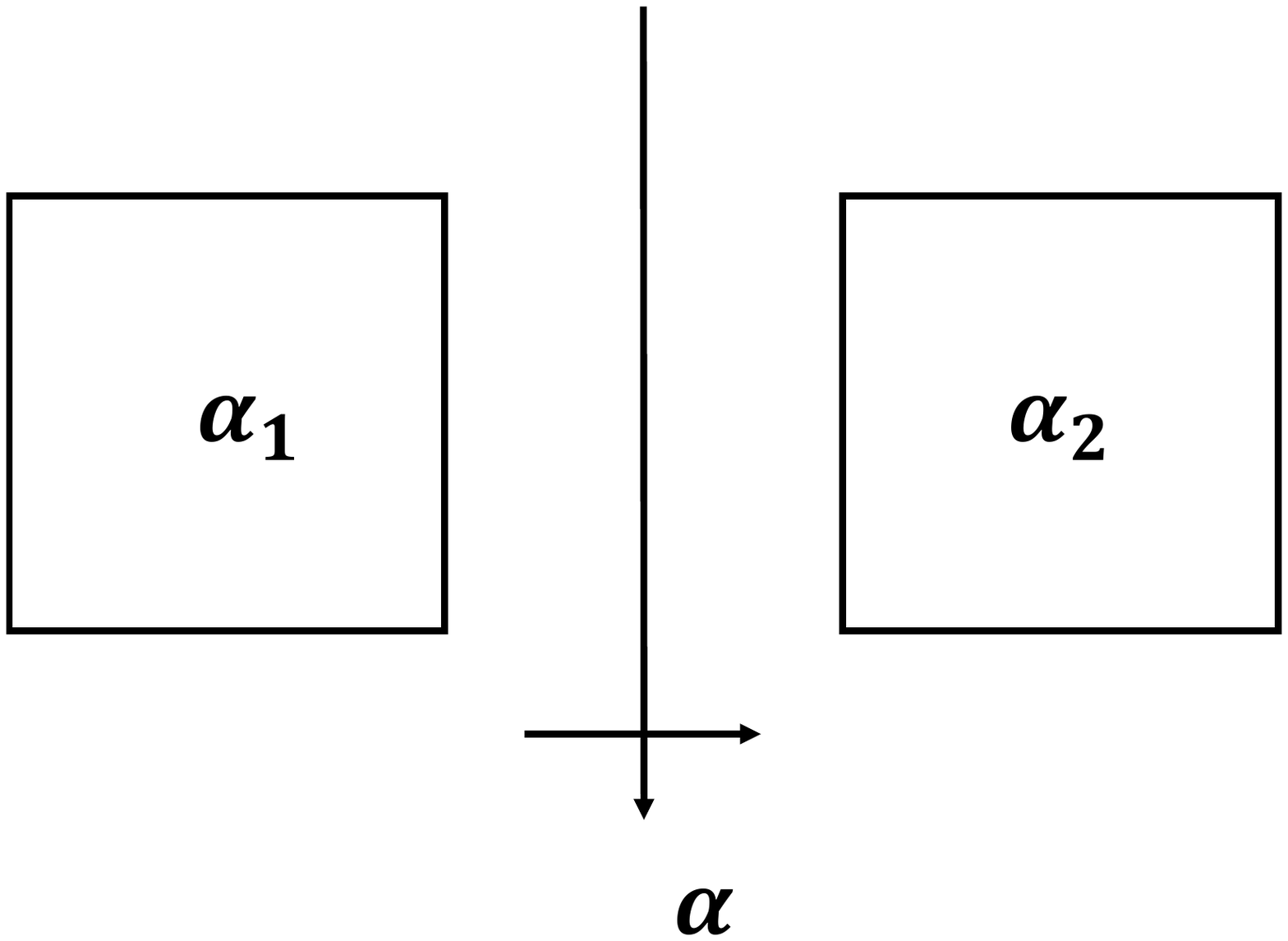}
\end{center}
\vspace{-13mm}
\caption{The coloring condition C3} 
\label{xy-col}
\end{figure}

\begin{figure}[H]
\begin{center}
\includegraphics[width=12cm, height=8cm]{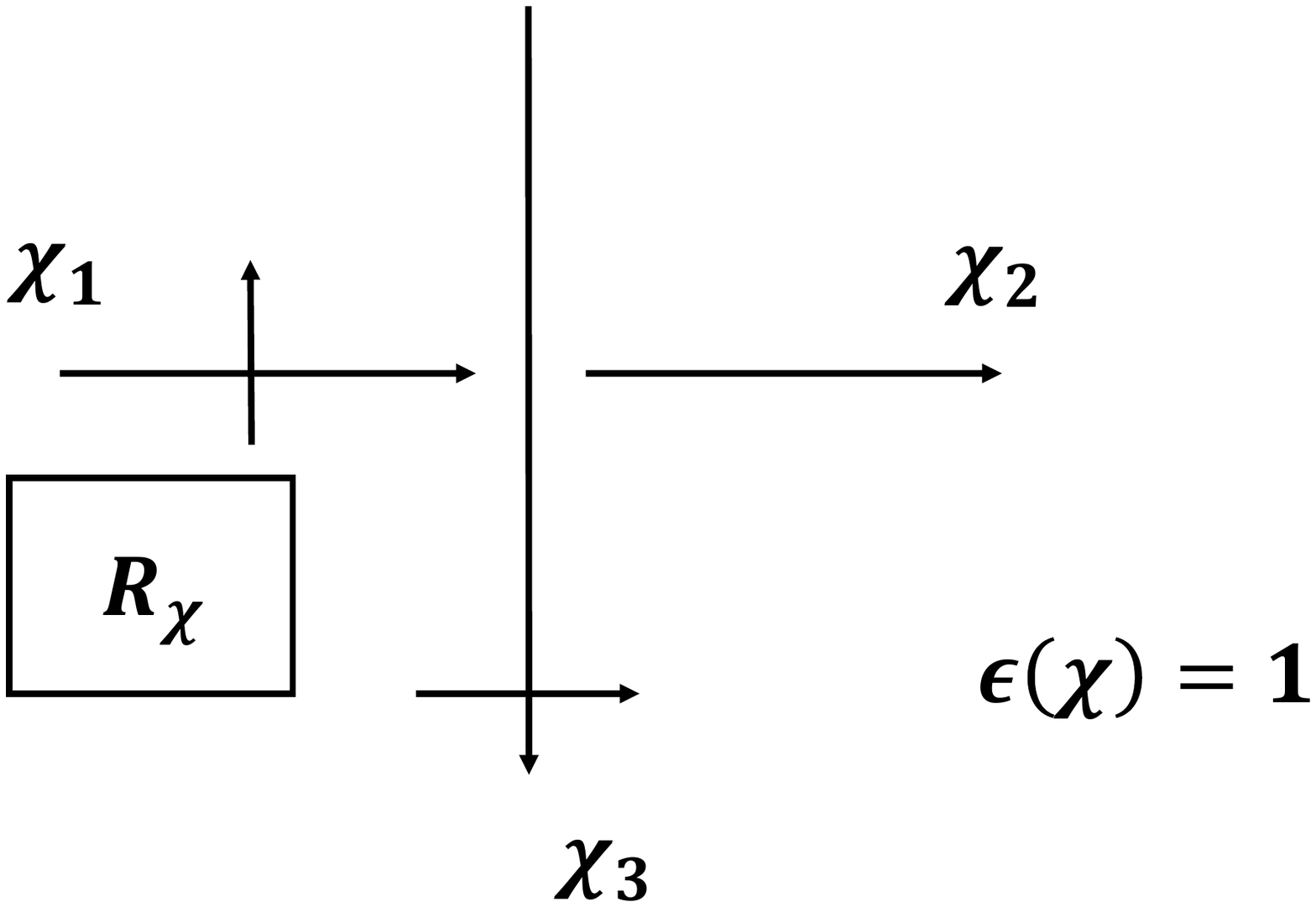}
\end{center}
\vspace{-13mm}
\caption{Weight of a crossing $\chi$}
\label{weight}
\end{figure}

Using a $G$-family of quandles and an Abelian group $A$, we define a chain complex, denoted by $C_{\ast}(X)_{Y}$, and the cochain complex, denoted by $C^{\ast}(X;A)_{Y}:=\text{Hom}(C_{\ast}(X)_{Y},A)$. 
Then, we also define the associated homology groups and cohomology groups (refer to \cite{ishii1}). 
For an $X_{Y}$-coloring $C$ and a crossing $\chi$ of a diagram $D$ of an oriented spatial trivalent graph of a handlebody-link, we define the weight of the crossing $\chi$ by $w(\chi; C):=\epsilon(\chi)(C(R_{\chi}), C(\chi_{1}), C(\chi_{3}))$, where $\chi_{1}$ is an under-arc such that the orientation of $\chi_1$ points into the crossing $\chi$, $\chi_{3}$ is an over-arc whose normal orientation points from $\chi_1$ to the other arc $\chi_2$,  $R_{\chi}$ is a complementary region such that normal orientations of $\chi_1$ and $\chi_3$ point to the opposite regions with respect to $\chi_1$ and $\chi_3$ (Figure \ref{weight}), and $\epsilon(\chi)$ is the sign of the crossing $\chi$.

Concerning homology
theory of $G$-families of quandles, the following lemma 
is known (see \cite{ishii1}).

\begin{lemm}
Let $(X, \{\ast_{g}\}_{g\in G})$ and $(Y, \{\bar{\ast}_{g}\}_{g\in G})$ be a $G$-family of quandles and an $X$-set, respectively.
Let $D$ be a diagram of an oriented spatial trivalent graph of a handlebody-link. 
Let $C$ be an $X_{Y}$-coloring of $D$.
Then, the sum of the weights $W(D;C):=\sum_{\chi\in D}w(\chi; C)$ is a $2$-cycle of $C_{\ast}(X)_{Y}$.
\end{lemm}

Let $A$ be an Abelian group.
Let $\theta$ be a 2-cocycle of the cochain complex $C^{\ast}(X;A)_{Y}$.
We define the multiset $\Phi_{\theta}(D)$ as follows:
\[
\Phi_{\theta}(D):=\{\theta(W(D;C))\in A\mid C\in\text{Col}_{X}(D)_{Y}\}.
\]

Concerning cohomology theory of $G$-families of quandles, the following theorem is also known (see \cite{ishii1}).

\begin{theo}\label{coho}
Let $(X, \{\ast_{g}\}_{g\in G})$ and $(Y, \{\bar{\ast}_{g}\}_{g\in G})$ be a $G$-family of quandles and an $X$-set, respectively.
Let $H$ be a handlebody-link and $D$ be a diagram of an oriented spatial trivalent graph of $H$.
Let $A$ and $\theta$ be an Abelian group and a $2$-cocycle of the cochain complex $C^{\ast}(X;A)_{Y}$, respectively.
Then, 
$\Phi_{\theta}(D)$ does not depend on the choice of $D$ and is an invariant of the handlebody-link $H$.
\end{theo}

We can write the multiset
$\Phi_{\theta}(D)$ in the form 
\[
\Phi_{\theta}(D)=\{(a_{1})_{l_{1}}, \ldots, (a_{m})_{l_{m}}\},
\]
where 
$l_{j}$ is the multiplicity of 
$a_{j}\in A$, and 
$(a_{j})_{l_{j}}$ represents 
$\underbrace{a_{j}, \ldots, a_{j}}_{l_j \text{-times}}$.
Using these notations and a natural number $N$, we define the set $\Phi_{\theta}(D)_{N}$ as follows:
\[
\Phi_{\theta}(D)_{N}:=\{(a_{1}, l_{1}/N), \ldots, (a_{m}, l_{m}/N)\mid (a_{i}, l_{i}/N)\in A\times \mathbb{Q}\}.
\]

\section{Main results}
In this section, we construct algebraic invariants of embedded surfaces in $S^3$.
We first use $X$-colorings combined with Lemma \ref{lem:2}.
Then, we will also use another quandle invariant given in Theorem \ref{coho}.

Let $F_{1}$ and $F_{2}$ be surfaces in $S^3$.
Two surfaces $F_1$ and $F_2$ are said to be {\it isotopic}, denoted by $F_{1}\cong F_{2}$, if there exists an isotopy $f_{t}:S^{3}\to S^{3}$, $t\in[0,1]$ such that $f_{0}=\text{id}_{S^3}$ and $f_{1}(F_{1})=F_{2}$.
We denote by $V_{i}$ and $W_{i}$ the closures of the connected components of $S^{3}\setminus F_i$, ($i=1, 2$).
Let $V_{i}=H_{V_{i}}\cup C_{V_{i}}$ and $W_{i}=H_{W_{i}}\cup C_{W_{i}}$ be Heegaard splittings of $V_{i}$ and $W_{i}$, respectively, consisting of handlebodies $H_{V_{i}}$, $H_{W_{i}}$ and compression bodies $C_{V_{i}}$, $C_{W_{i}}$.
We call the $2$-component handlebody link $L_{i}:=H_{V_{i}}\sqcup H_{W_{i}}$ an \emph{associated $2$-component handlebody-link} of $F_{i}$.

\begin{theo}\label{thm:6}
Let $G$ be a finite group and $(X, \{\ast_g\}_{g \in G})$ a $G$-family of quandles, where $X$ is a finite set.
Let $F$ be a closed connected orientable surface in $S^3$.
We denote by $V_{F}$ and $W_{F}$ the connected components of the exterior of $F$.
Let $F_{V}$ and $F_{W}$ be Heegaard surfaces of  $V_{F}$ and $W_{F}$, and $H_{V}$ and $H_{W}$ the corresponding handlebodies, respectively. We denote by $D_{V}$ and $D_{W}$ diagrams of oriented spatial trivalent graphs representing $H_{V}$ and $H_{W}$, respectively.
Then, the unordered pair 
\[
\left(\frac{\#{\rm Col}_{X}(D_{V})}{(\# G)^{g(F_{V})}} , \frac{\#{\rm Col}_{X}(D_{W})}{(\# G)^{g(F_{W})}}\right)
\]
of rational numbers is an isotopy invariant of $F$.
\end{theo}

\begin{proof}
Let $(V_{F},F_{V})$ and $(V_{F},F_{V}^{'})$ be Heegaard splittings of $V_{F}$.
Then, we have the corresponding pairs of a handlebody-knot and a compression body in $S^3$, say $(H_{V},C_{V})$ and $(H_{V}^{'},C_{V}^{'})$, respectively.
By the Reidemeister--Singer theorem, the two Heegaard splittings $(V_{F},F_{V})$ and $(V_{F},F_{V}^{'})$ become equivalent to the same Heegaard splitting $(V_F, \bar{F}_V)$ after $m$ times and $n$ times stabilizations, respectively, for
some non-negative integers $m$ and $n$. 
Let $\bar{H}$ be the handlebody-knot and $\bar{C}_V$ be the compression body corresponding to $(V_F, \bar{F}_V)$.
Let $\bar{D}_{V}$ be a diagram of an oriented spatial trivalent graph representing $\bar{H}_{V}$.
Then, by Lemma \ref{lem:2}, we observed a variation of the cardinality of the set of $X$-colorings after applying a stabilization.
Then, we have $\# \text{Col}_{X}(\bar{D}_{V})=\# \text{Col}_{X}(D_{V}) \cdot (\# G)^{m}=\# \text{Col}_{X}(D_{V}^{'}) \cdot (\# G)^{n}$.
On the other hand, we have $g(\bar{F}_{V})=g(F_{V})+m=g(F_{V}^{'})+n$.
Hence we see 
$$
\left(\frac{\# {\rm Col}_X(\bar{D}_V)}{(\# G)^{g(\bar{F}_V)}}\right) = \left(\frac{\# {\rm Col}_X(D_V)}{(\# G)^{g(F_V)}}\right) = \left(\frac{\# {\rm Col}_X(D^{'}_V)}{(\# G)^{g(F^{'}_V)}}\right).
$$
Combining Theorem \ref{thm:4}, we have the desired conclusion.
\end{proof}

In Theorem \ref{thm:6}, we used each connected component of a $2$-component handlebody-link independently.
Then we constructed an invariant of handlebody-knots up to stabilizations.
However, we can show the following even if we use an invariant of $2$-component handlebody-links up to stabilizations.

\begin{theo}\label{link}
Let $G$ be a finite group and $(X, \{\ast_g\}_{g \in G})$ a $G$-family of quandles, where $X$ is a finite set.
Let $F$ be a closed connected orientable surface in $S^3$.
We denote by $V_{F}$ and $W_{F}$ the connected components of the exterior of $F$.
Let $F_{V}$ and $F_{W}$ be Heegaard surfaces of  $V_{F}$ and $W_{F}$, and $H_{V}$ and $H_{W}$ the corresponding handlebodies, respectively. 
Let $L=H_{V}\sqcup H_{W}$ be the $2$-component handlebody-link.
We denote by $D$ a diagram of an oriented spatial trivalent graph representing $L$.
Then, the rational number 
\[
\frac{\#{\rm Col}_{X}(D)}{(\# G)^{g(F_{V})+g(F_{W})}}
\]
is an isotopy invariant of $F$.
\end{theo}

By using an argument similar to that used in Theorem \ref{thm:6}, we can also construct another algebraic invariant of embedded surfaces in $S^3$.
In order to construct such an invariant, we use Theorem \ref{thm:4} and \ref{coho}.

\begin{theo}
Let $G$ be a finite group and $(X, \{\ast_g\}_{g \in G})$ a $G$-family of quandles, where $X$ is a finite set.
Let $Y$ be an $X$-set, where $Y$ is a finite set.
Let $F$ be a closed connected orientable surface in $S^3$.
We denote by $V_{F}$ and $W_{F}$ the connected components of the exterior of $F$.
Let $F_{V}$ and $F_{W}$ be Heegaard surfaces of  $V_{F}$ and $W_{F}$, and $H_{V}$ and $H_{W}$ the corresponding handlebodies, respectively. 
Let $L=H_{V}\sqcup H_{W}$ be a $2$-component handlebody-link.
We denote by $D$ a diagram of an oriented spatial trivalent graph representing $L$.
Then, the multiset
$\Phi_{\theta}(D)_{(\# G)^{g(F_{V})+g(F_{W})}}$ is an isotopy invariant of $F$.
\end{theo}

\begin{proof}
Let $L^{'}$ be a $2$-component handlebody-link obtained from $L$ by applying a stabilization, and we denote by $D^{'}$ a diagram of an oriented spatial trivalent graph representing $L^{'}$.
By the definition of a stabilization of handlebody-links, attached solid tori are separated, then the condition of crossings of $D^{'}$ is same as that of $D$.
Moreover, by using an argument similar to that used in the proof of Lemma \ref{lem:2}, we can show that $\#{\rm Col}_{X}(D^{'})_{Y}=\#{\rm Col}_{X}(D_{V})_{Y} \cdot \# G$.
Then, we also have an invariance of the weight of each crossing.
Therefore, we obtain the required conclusion.
\end{proof}

For a 2-component handlebody-link, its linking number is an invariant  (refer to \cite{miz}).
We have a $2$-component handlebody-link from a closed connected orientable surface in $S^3$ by considering Heegaard splittings of the connected components of the exterior of the surface.
By the definition of the stabilization of handlebody-links, we can show that linking number does not change after applying a stabilization.
Then, the linking number of a 2-component handlebody-link obtained from $F$ is also an isotopy invariant of $F$.

\section{Examples}
We compute our invariants given in Theorem \ref{thm:6} for the following examples of bi-knotted surfaces (see Definition \ref{surfaces}).
Throughout this section, we set $X=\mathbb{Z}/3\mathbb{Z}$, $G=\mathbb{Z}/2\mathbb{Z}$, $g\ast_{0}h:=g$, $g\ast_{1}h=2h-g$ for any $g,h\in \mathbb{Z}/3\mathbb{Z}$ and $0,1\in \mathbb{Z}/2\mathbb{Z}$.
Then, it is known that $(X, \{\ast_g\}_{g \in G})$ is a $G$-family of quandles (\cite{ishii1}).

In order to construct bi-knotted surfaces, we consider two mutually disjoint $3$-balls $B_1$ and $B_2$, respectively.
Then we take properly embedded arcs $\alpha_1$ and $\alpha_2$ from $B_1$ and $B_2$, respectively.
We assume that $\alpha_1$ is unknotted and $\alpha_2$ is knotted corresponding to the trefoil.
We denote by $N(\alpha_{1})$ and $N(\alpha_{2})$ regular neighborhoods of $\alpha_1$ and $\alpha_2$.
Then we remove regular neighborhoods $N(\alpha_{1})$ and $N(\alpha_{2})$ from $B_1$ and $B_2$.
We set $B^{'}_{1}={\rm cl}(B_{1}\setminus N(\alpha_{1}))$ and $B^{'}_{2}={\rm cl}(B_{2}\setminus N(\alpha_{2}))$.
Then we connect $B^{'}_{1}$ and $B^{'}_{2}$ by the $1$-handle $D^{1}\times D^{2}$ so that $\{-1\}\times D^{2}$ is attached to $\partial B^{'}_{1}$, $\{1\}\times D^{2}$ is attached to $\partial B^{'}_{2}$, and the attached $1$-handle throughs straightforwardly both regions which are obtained from $B_{1}$ and $B_2$ by removing $N(\alpha_{1})$ and $N(\alpha_{2})$.
We assume that the attached $1$-handle admits a tangle induced from that of $N(\alpha_{2})$.
We set $F$ as the boundary of the the resulting compact connected orientable $3$-manifold $M$ (see Figure \ref{fig:one}).
Similarly, we construct compact connected orientable $3$-manifold $M^{'}$ by using an unknotted arc and a knotted arc corresponding to the figure eight knot.
Then we set $F^{'}$ as the boundary of the $3$-manifold (see Figure \ref{fig:two}).

We assume that $M$ is embedded in $S^3$.
Then, $M$ is one of the connected components of the exterior of $F=\partial M$, say, $V_F$ (see Figure \ref{fig:three}).
Then, we obtain the other connected components of the exterior of $F$ by the following procedures.
We first consider the closure of the exterior of $B_{1}^{'}$ in $S^3$, which is obtained from a $3$-ball $B^3$ by attaching $N(\alpha_{1})$.
Then, we remove the interior of $B_{2}$ from $B^3$ and attach $N(\alpha_{2})$ to $\partial B_{2}$.
Finally, we remove the $2$-handle $D^{2}\times D^{1}$ so that $D^{2}\times \{-1\}$ is attached to $\partial B^3$,  $D^{2}\times \{1\}$ is attached to $\partial B_{2}$, and the $2$-handle throughs both $N(\alpha_{1})$ and $N(\alpha_{2})$ straightforwardly.
The resulting compact connected orientable $3$-manifold, say $W_F$, is the other connected component of $F$ (see Figure \ref{fig:four}).
We set $M^{'}=V_{F}^{'}$ (see Figure \ref{fig:seven}).
Then, by a similar argument, we obtain the other connected component $W_{F}^{'}$ of $F^{'}$ (see Figure \ref{fig:eight}).

Moreover, by removing 2-handles from $V_{F}$ and from $W_F$, respectively, we obtain two handlebody-knots $H_{V}$ and $H_{W}$ (Figures \ref{fig:five} and \ref{fig:six}).
Then, by projecting their oriented spatial trivalent graphs on $\mathbb{R}^{2}$, we obtain diagrams $D_{V}$ and $D_{W}$ of $H_{V}$ and $H_{W}$ (Figure \ref{D_V} and Figure \ref{D_W}).
By using a similar argument, we obtain handlebody-knots $H_{F}^{'}$ and $H_{W}^{'}$ (Figures \ref{fig:nine} and \ref{fig:ten}).
Similarly, by projecting their oriented spatial trivalent graphs on $\mathbb{R}^{2}$, we obtain diagrams $D^{'}_{V}$ and $D^{'}_{W}$ of $H^{'}_{V}$ and $H^{'}_{W}$ (Figures \ref{D_V^{'}} and \ref{D_W^{'}})

\begin{exam}
For the surface $F$ (Figure \ref{fig:one}), we compute $\# \text{Col}_{X}(D_{V})$ and $\# \text{Col}_{X}(D_{W})$.
Then we have $\# \text{Col}_{X}(D_{V})=6$ and $\# \text{Col}_{X}(D_{W})=48$.
Therefore we obtain
\[
\frac{\# \text{Col}_{X}(D_{V})}{(\# (\mathbb{Z}/2\mathbb{Z}))^{3}}=\frac{3}{4},\ 
\frac{\# \text{Col}_{X}(D_{W})}{(\#(\mathbb{Z}/2\mathbb{Z}))^{3}}=6.
\]
\end{exam}

\begin{exam}
Let us now focus on the surface $F^{'}$ (Figure \ref{fig:two}).
We also compute $\# \text{Col}_{X}(D^{'}_{V})$ and $\# \text{Col}_{X}(D^{'}_{W})$.
Then we have $\# \text{Col}_{X}(D^{'}_{V})=6$ and $\# \text{Col}_{X}(D^{'}_{W})=24$.
Hence we have
\[
\frac{\# \text{Col}_{X}(D_{V^{'}})}{(\# (\mathbb{Z}/2\mathbb{Z}))^{3}}=\frac{3}{4},\ 
\frac{\# \text{Col}_{X}(D_{W^{'}})}{(\# (\mathbb{Z}/2\mathbb{Z}))^{3}}=3.
\]
\end{exam}

We can show that two surfaces $F$ and $F^{'}$ are not isotopic by using our invariant.

\begin{figure}[H]
  \begin{center}
   \includegraphics[width=15cm, height=10cm]{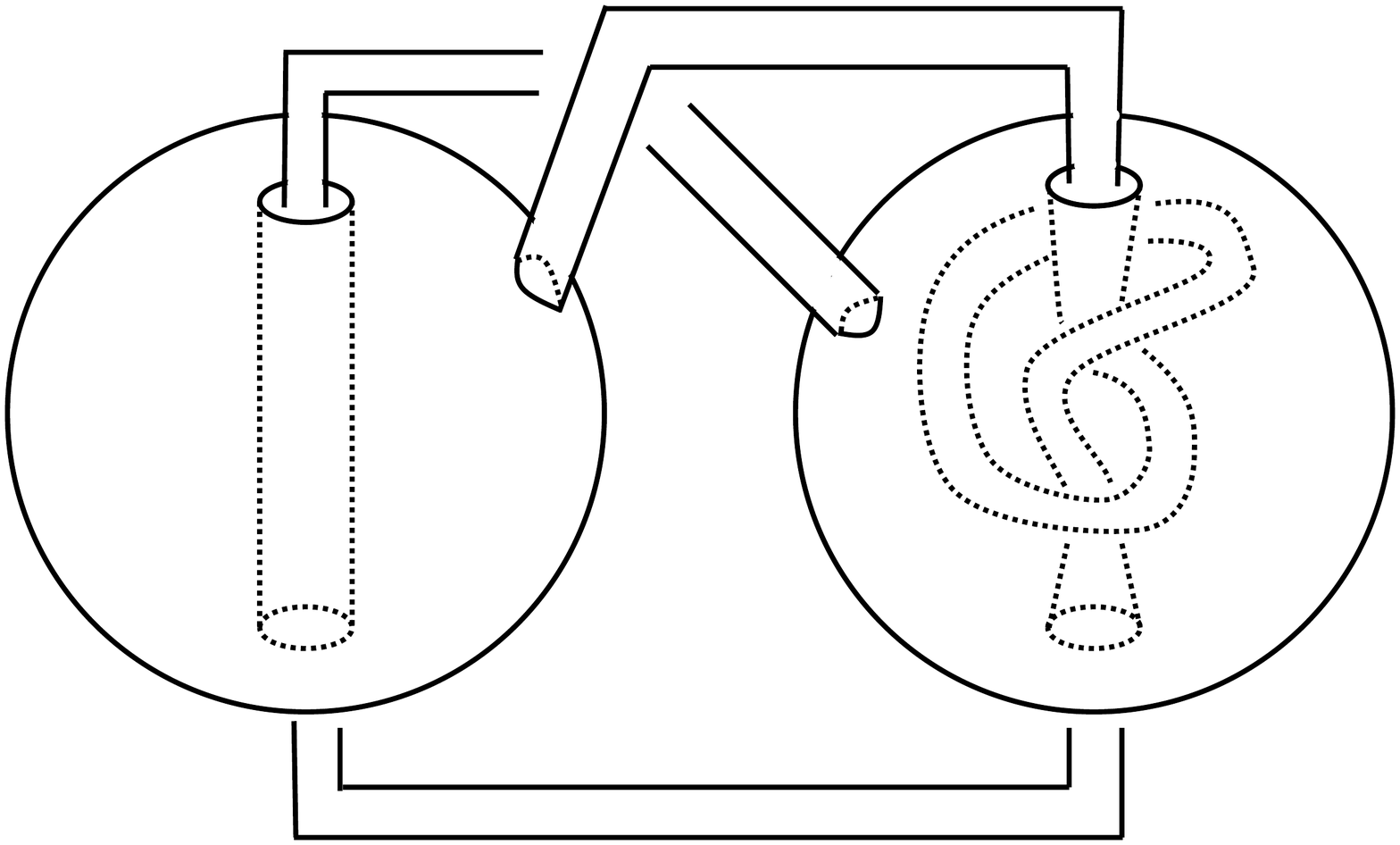}
  \end{center}
    \vspace{-13mm}
\caption{Bi-knotted surface $F$} 
  \label{fig:one}
\end{figure}
  
\begin{figure}[H]  
  \begin{center}
   \includegraphics[width=15cm, height=10cm]{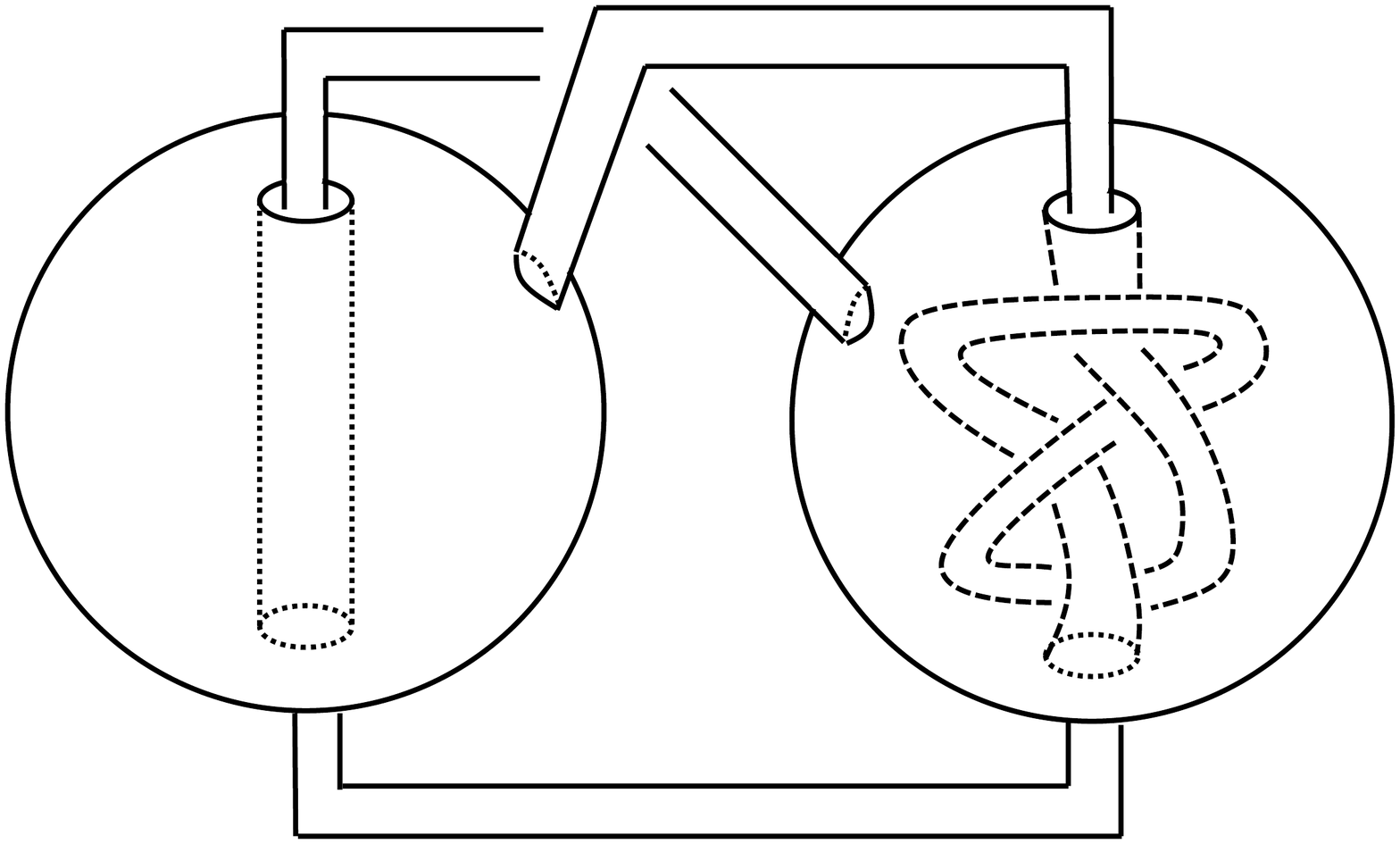}
  \end{center}
    \vspace{-13mm}
  \caption{Bi-knotted surface $F^{'}$}
  \label{fig:two}
\end{figure}

\begin{figure}[H]
  \begin{center}
   \includegraphics[width=15cm, height=10cm]{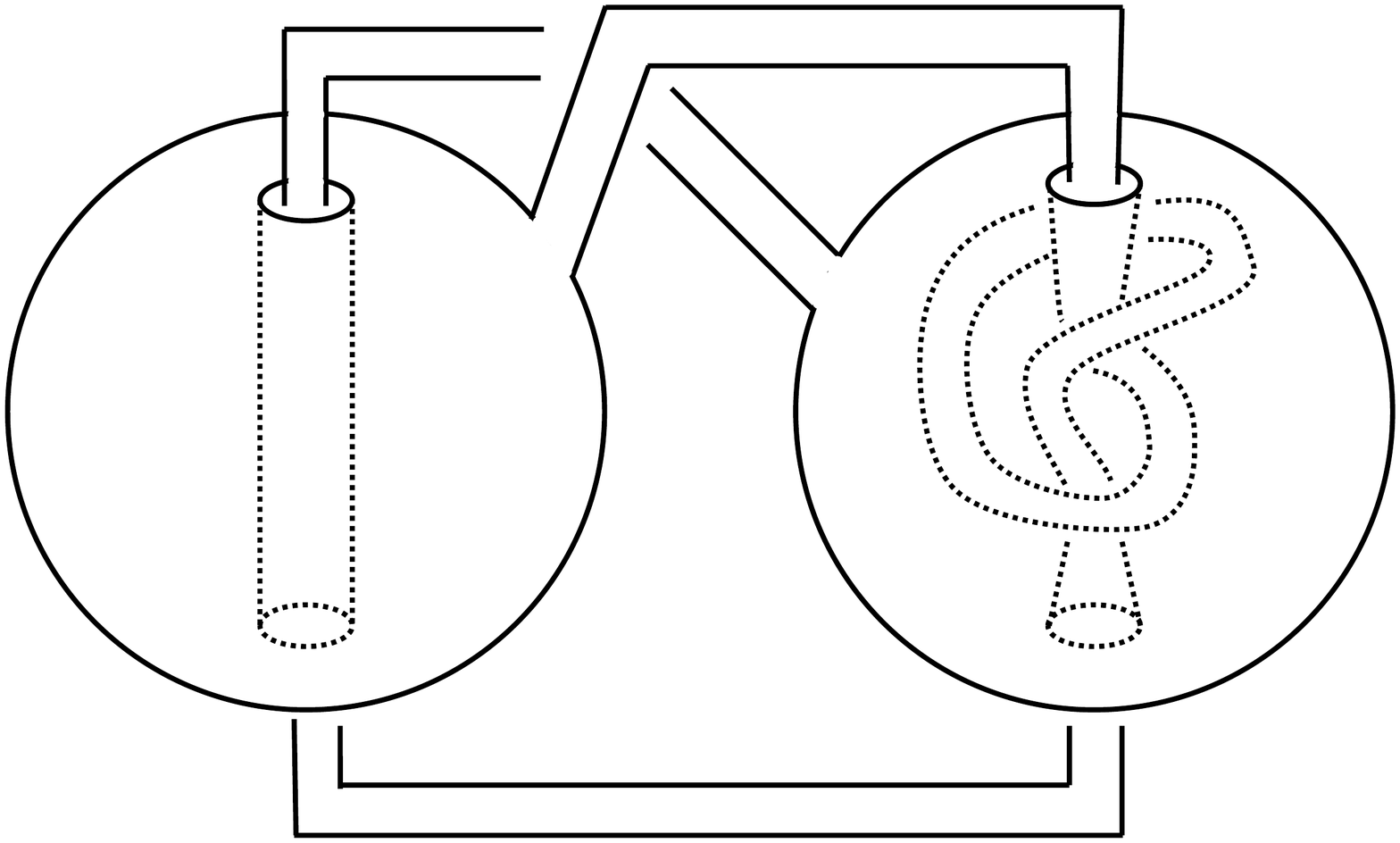}
  \end{center}
   \vspace{-13mm}
\caption{Exterior component $V_{F}$} 
  \label{fig:three}
\end{figure}  
  
\begin{figure}[H]  
  \begin{center}
   \includegraphics[width=15cm, height=10cm]{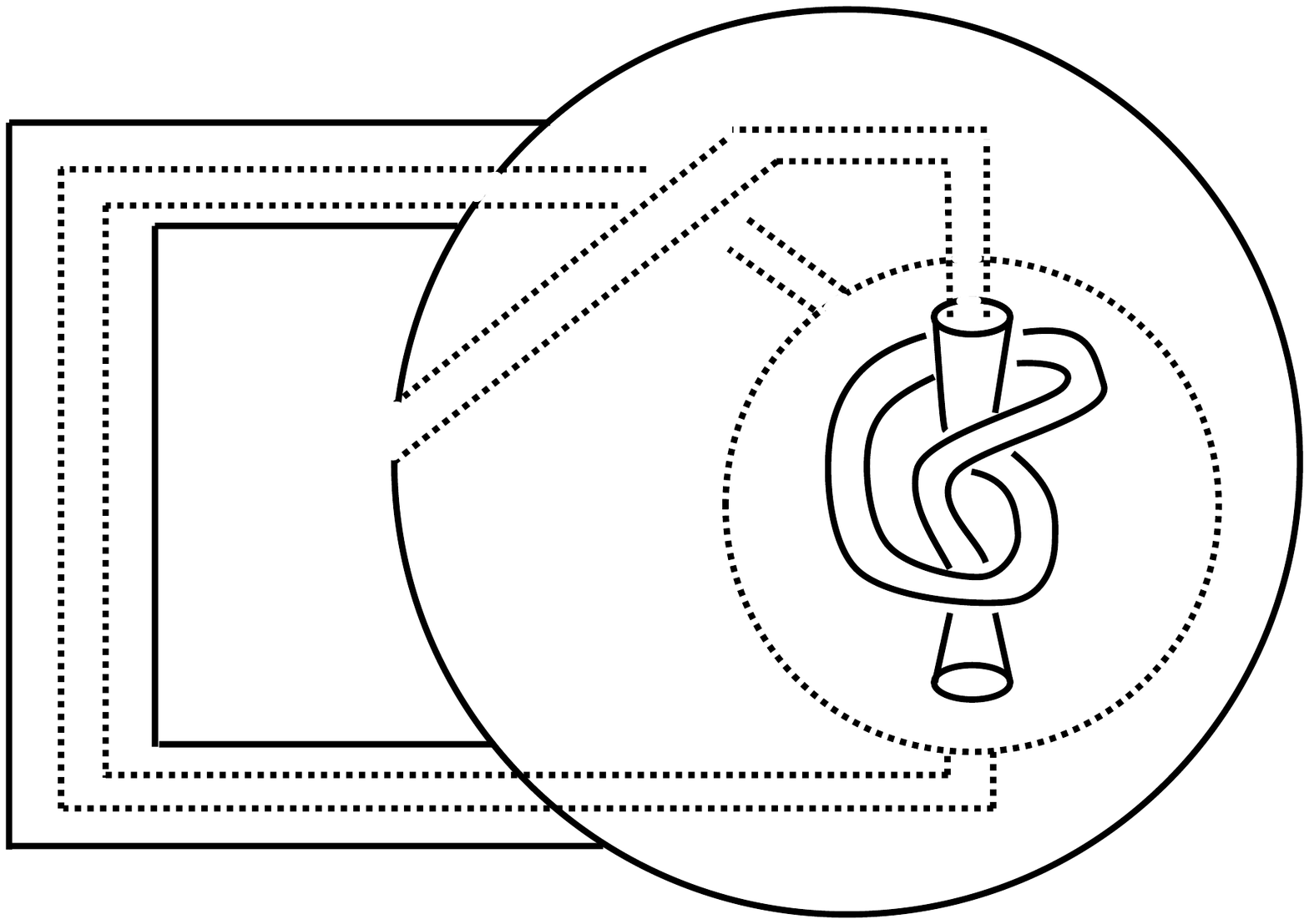}
  \end{center}
  \vspace{-13mm}
  \caption{Exterior component $W_{F}$}
  \label{fig:four}
\end{figure}

\vspace{-5mm}
\begin{figure}[H]
  \begin{center}
   \includegraphics[width=15cm, height=10cm]{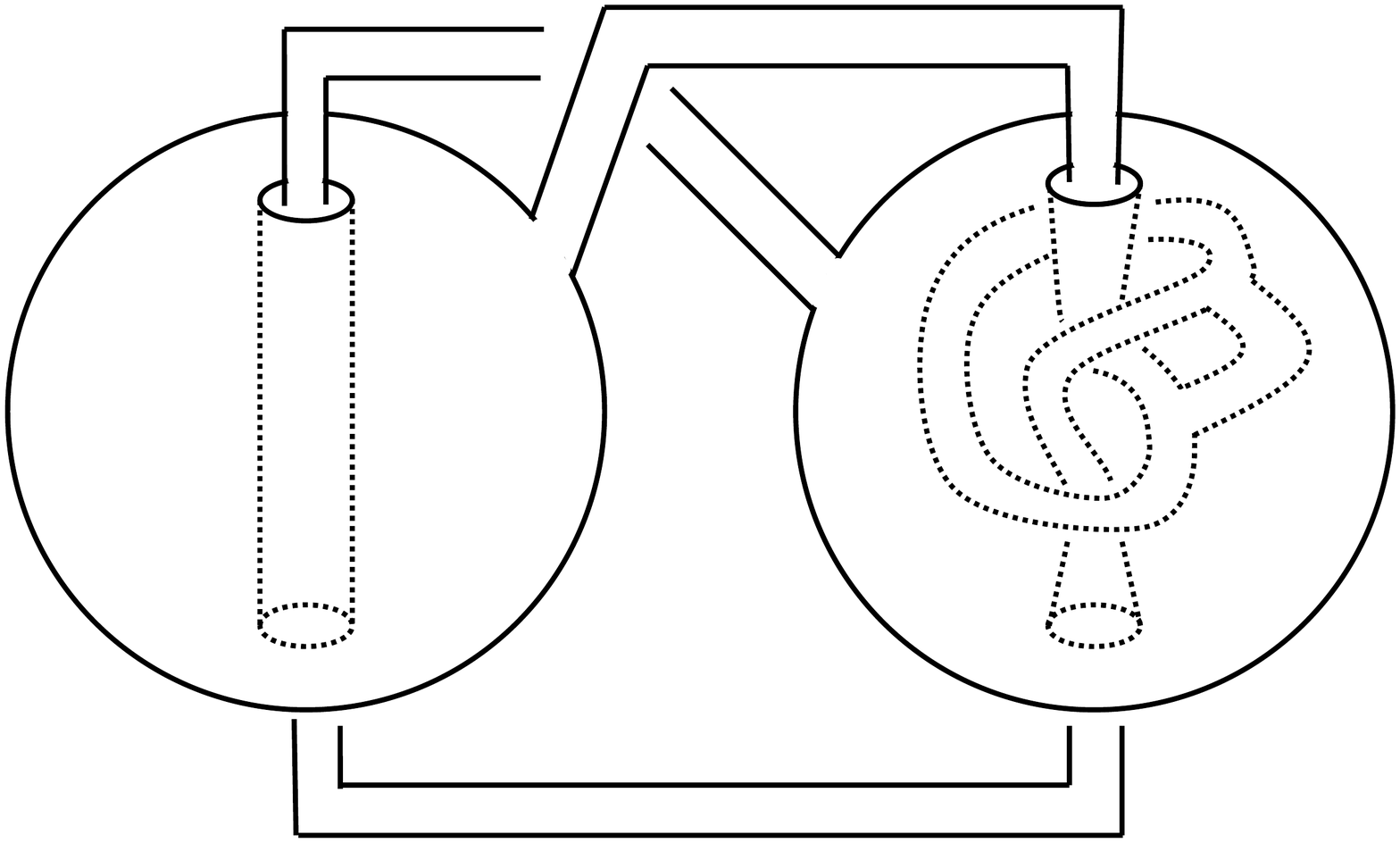}
  \end{center}
  \vspace{-13mm}
\caption{Handlebody-knot $H_{V}$} 
  \label{fig:five}
\end{figure}  
  
\begin{figure}[H]  
  \begin{center}
   \includegraphics[width=15cm, height=10cm]{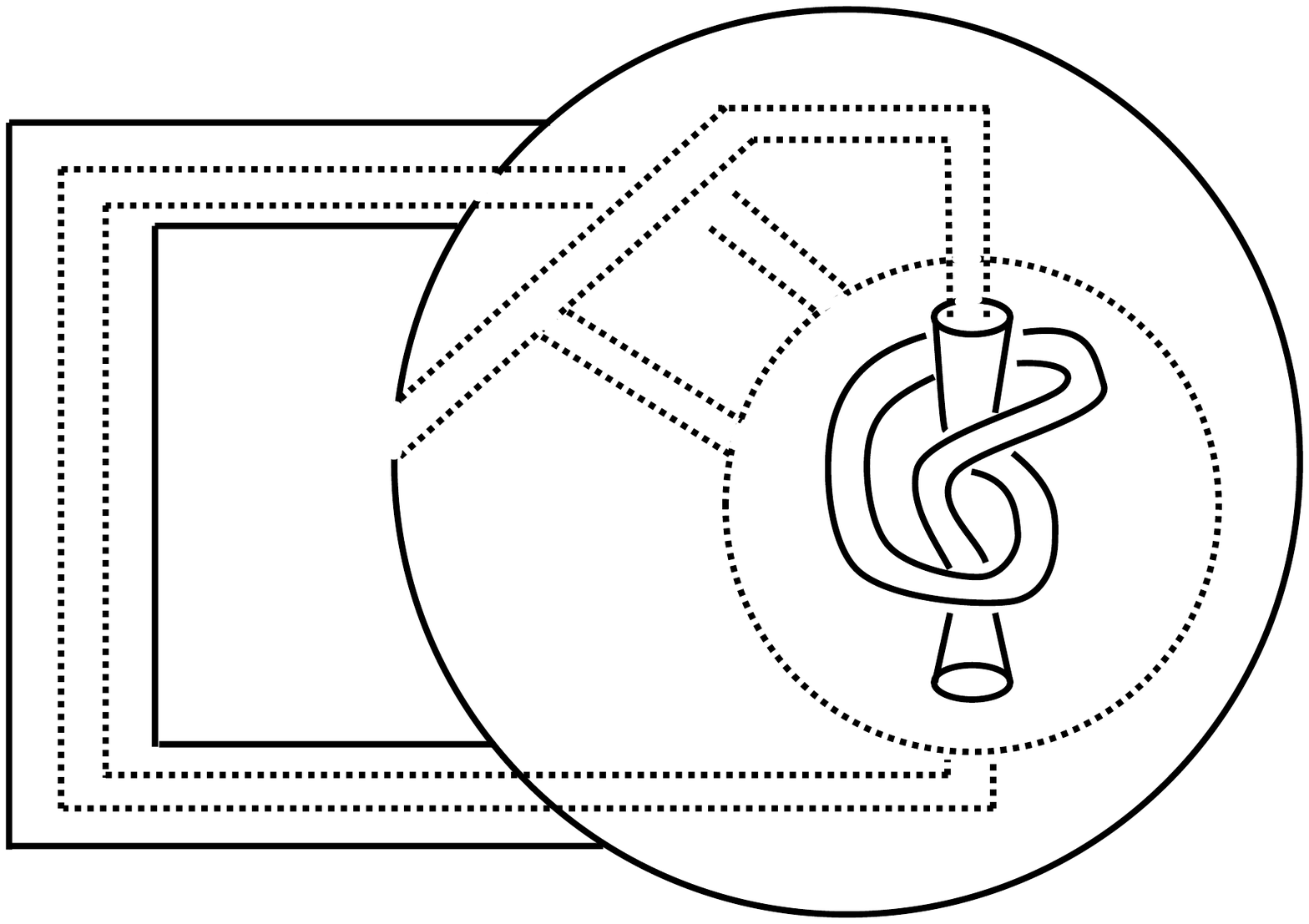}
  \end{center}
  \vspace{-13mm}
  \caption{Handlebody-knot $H_{W}$}
  \label{fig:six}
\end{figure}

\vspace{-5mm}
\begin{figure}[H]
  \begin{center}
   \includegraphics[width=15cm, height=10cm]{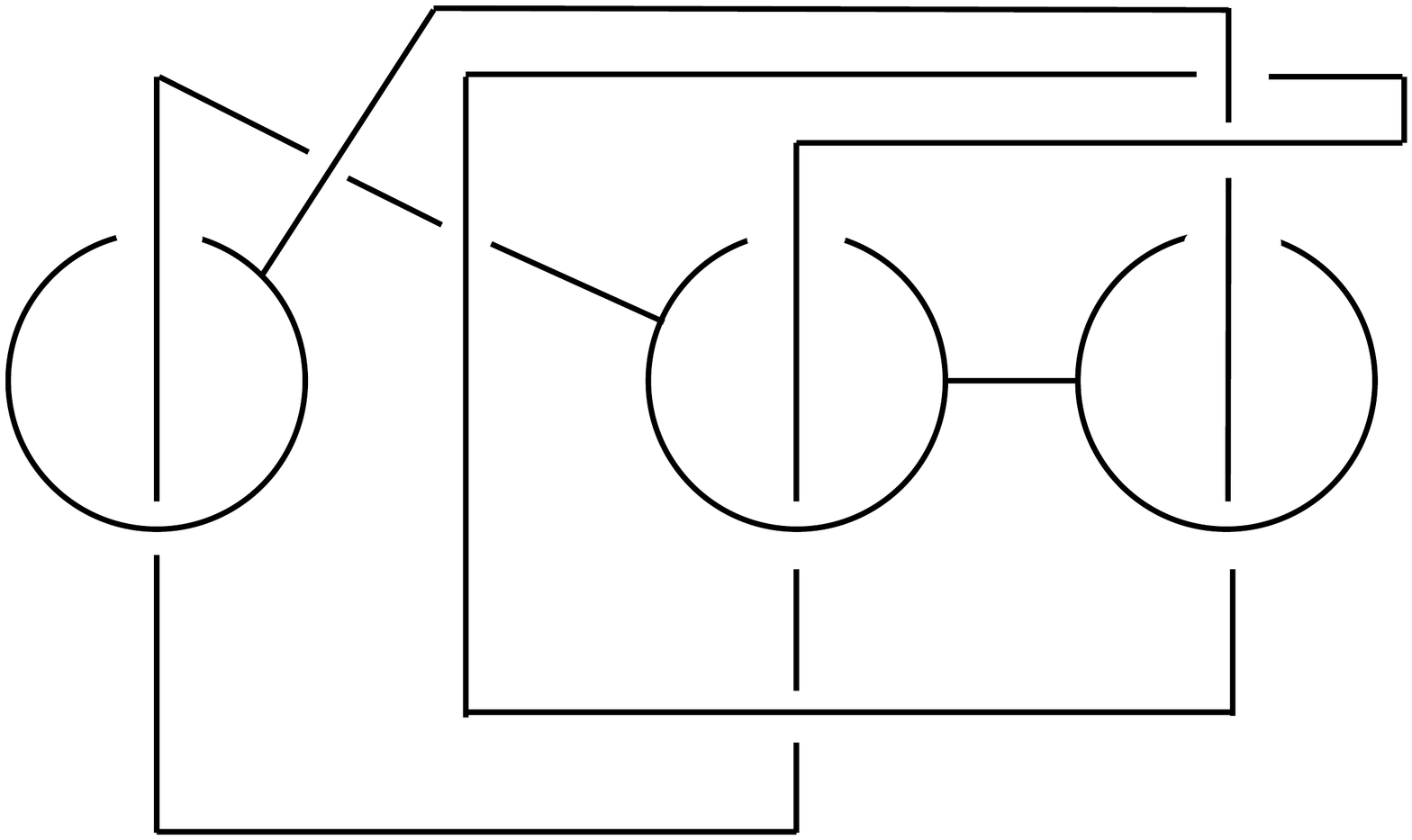}
  \end{center}
  \vspace{-13mm}
\caption{Diagram $D_{V}$ of the handlebody-knot $H_{V}$} 
  \label{D_V}
\end{figure}

\begin{figure}[H]
  \begin{center}
   \includegraphics[width=15cm, height=10cm]{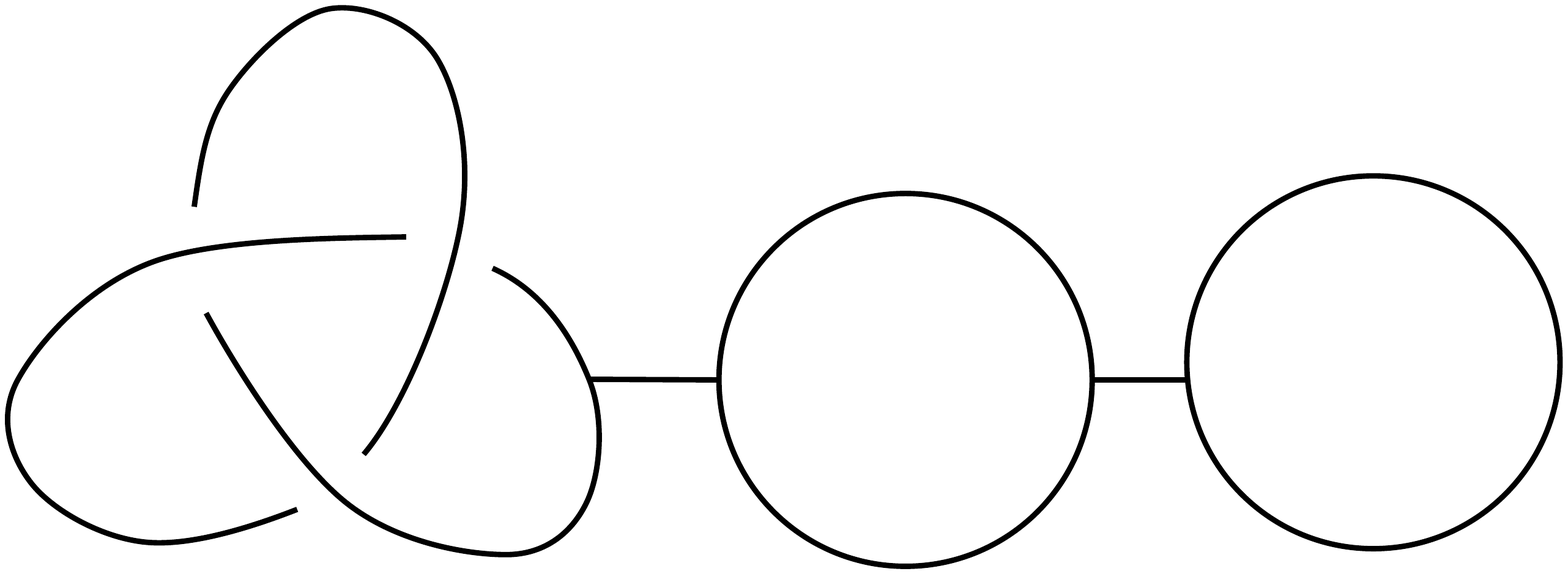}
  \end{center}
  \vspace{-13mm}
  \caption{Diagram $D_{W}$ of the handlebody-knot $H_{W}$}
  \label{D_W}
\end{figure}

\begin{figure}[H]
  \begin{center}
   \includegraphics[width=15cm, height=10cm]{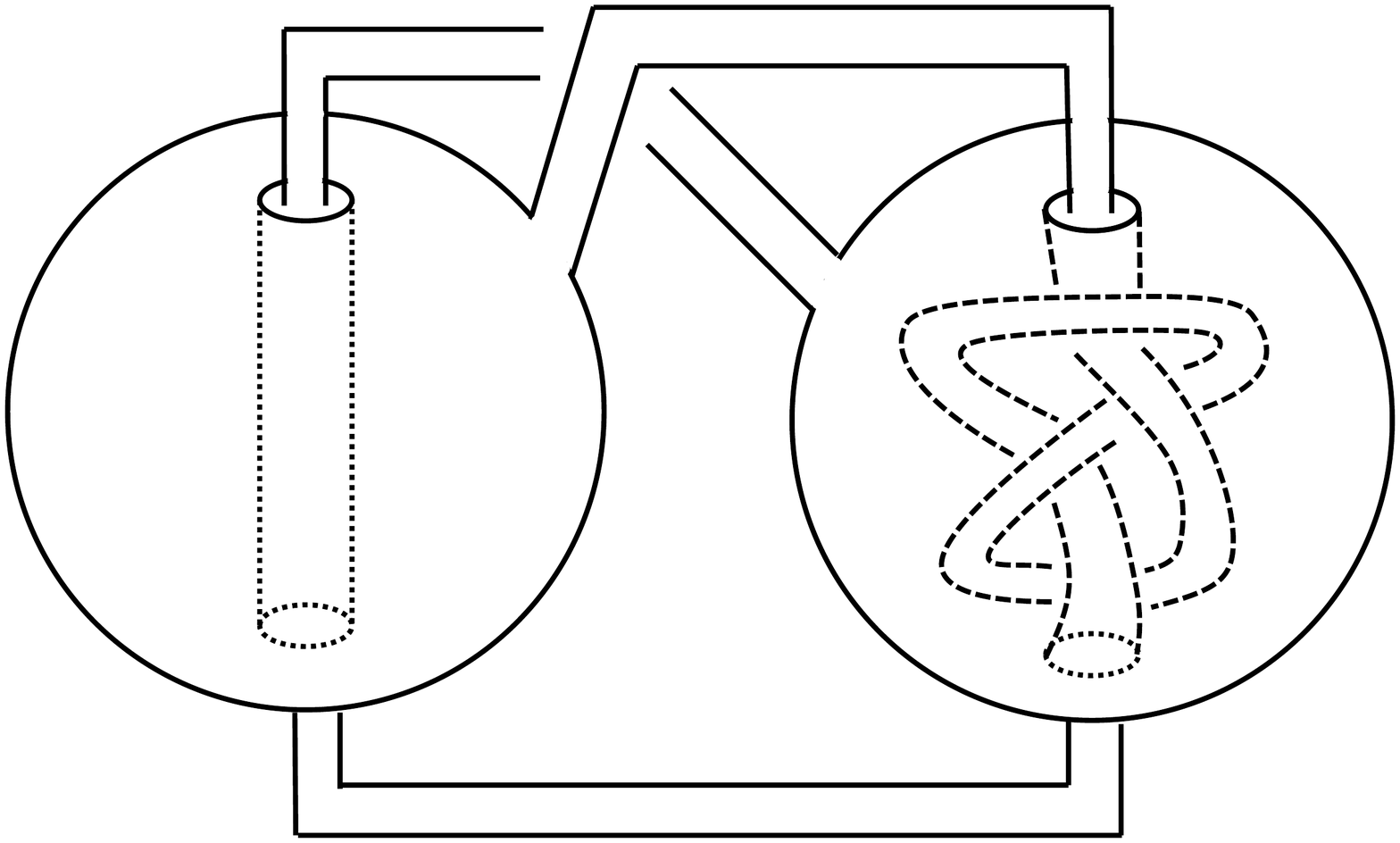}
  \end{center}
  \vspace{-13mm}
\caption{Exterior component $V_{F^{'}}$} 
  \label{fig:seven}
 \end{figure}  
  
\begin{figure}[H]  
  \begin{center}
   \includegraphics[width=15cm, height=10cm]{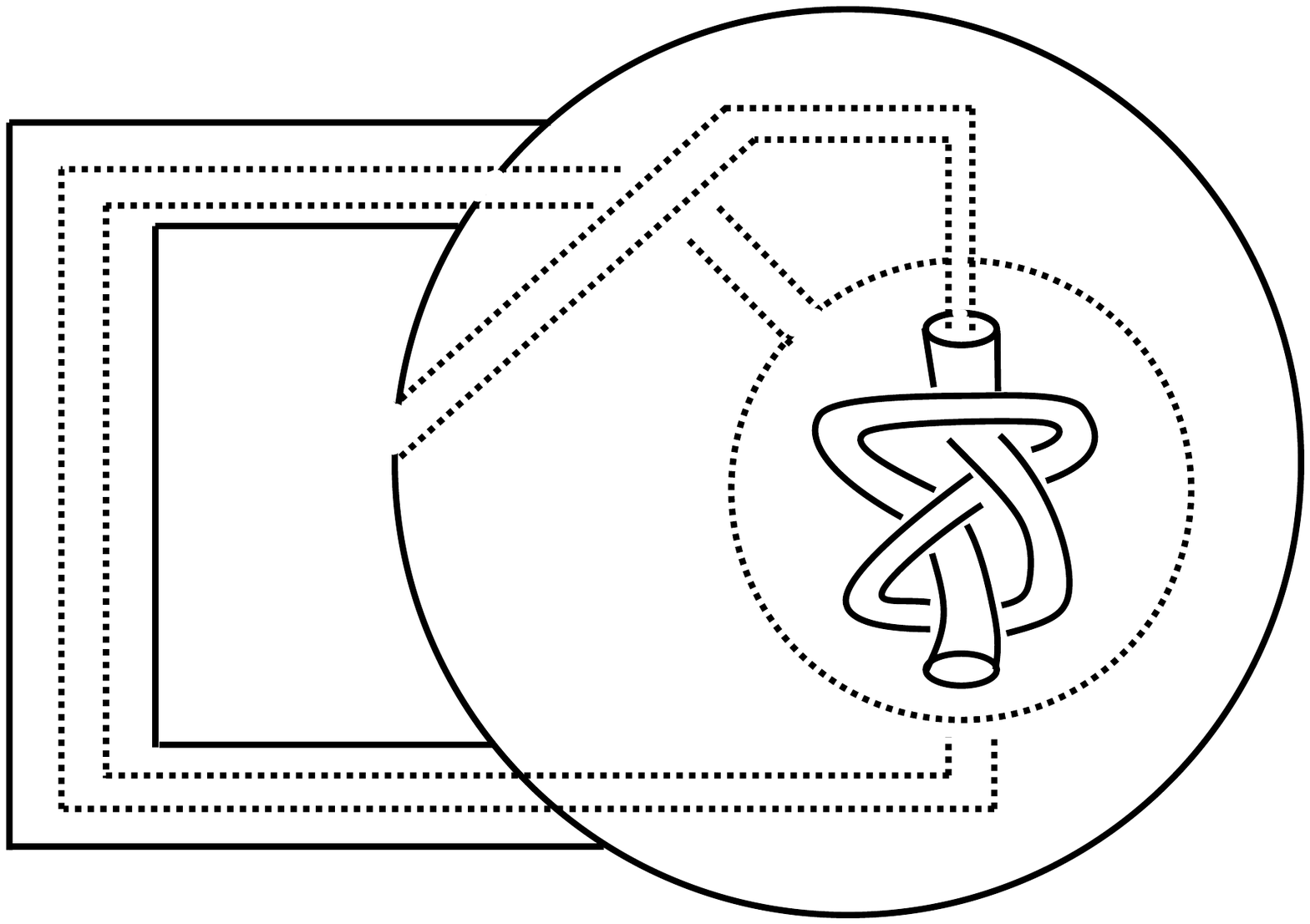}
  \end{center}
  \vspace{-13mm}
  \caption{Exterior component $W_{F^{'}}$}
  \label{fig:eight}
\end{figure}

\begin{figure}[H]
  \begin{center}
   \includegraphics[width=15cm, height=10cm]{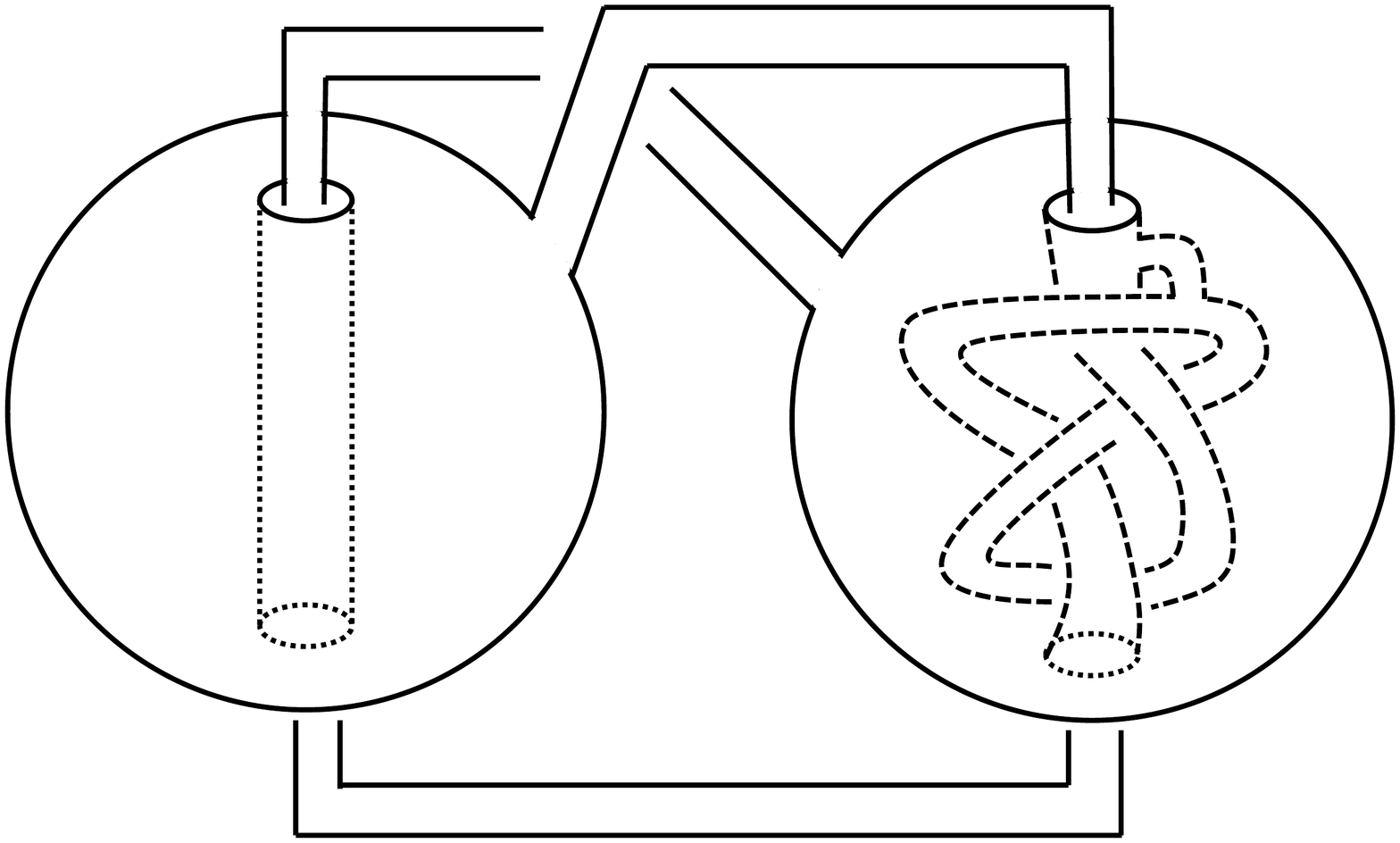}
  \end{center}
  \vspace{-13mm}
\caption{Handlebody-knot $H_{V^{'}}$} 
  \label{fig:nine}
 \end{figure}
  
 \begin{figure}[H]
  \begin{center}
   \includegraphics[width=15cm, height=10cm]{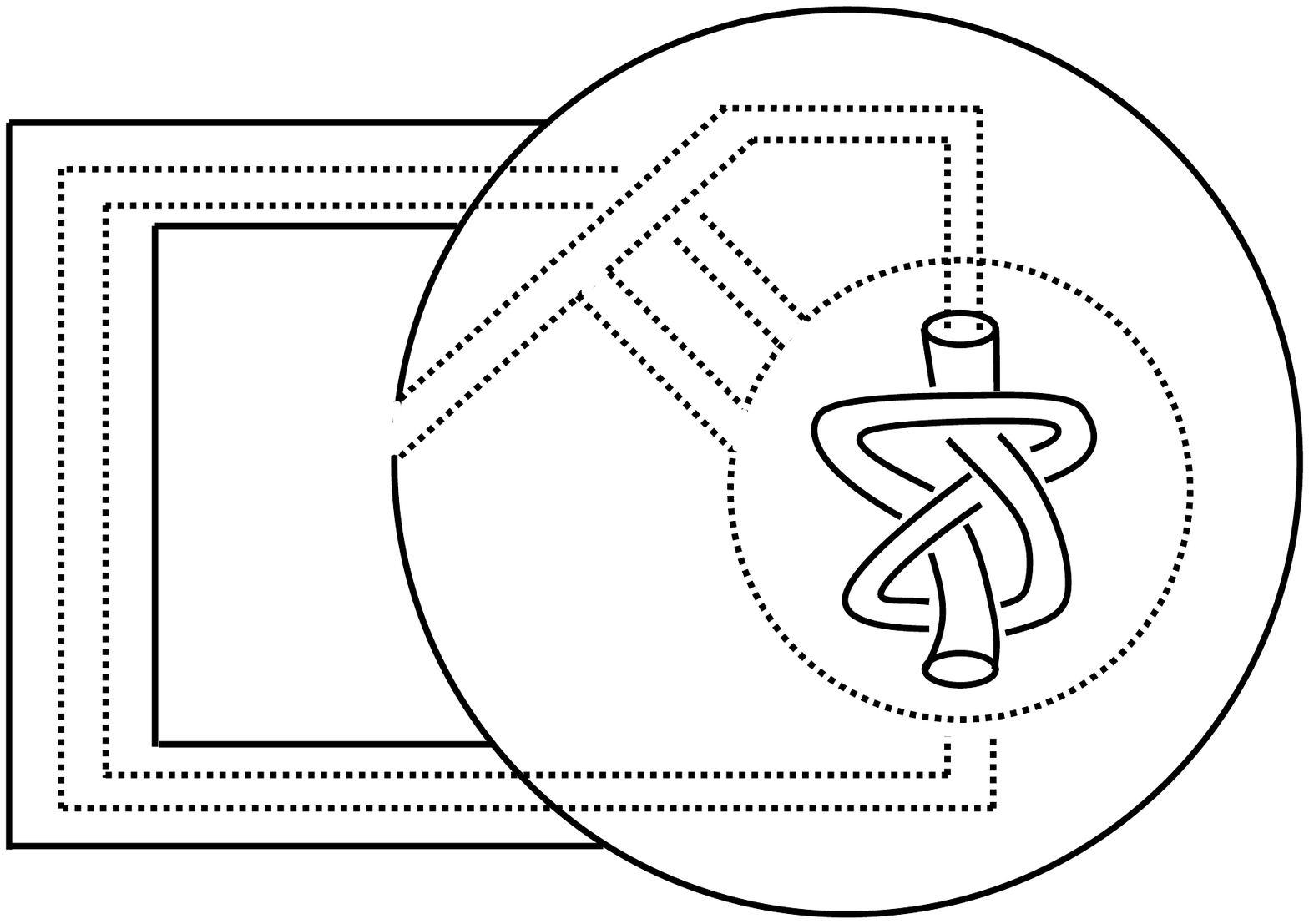}
  \end{center}
  \vspace{-13mm}
  \caption{Handlebody-knot $H_{W^{'}}$}
  \label{fig:ten}
\end{figure}

\vspace{-5mm}
\begin{figure}[H]
\begin{center}
\includegraphics[width=15cm, height=10cm]{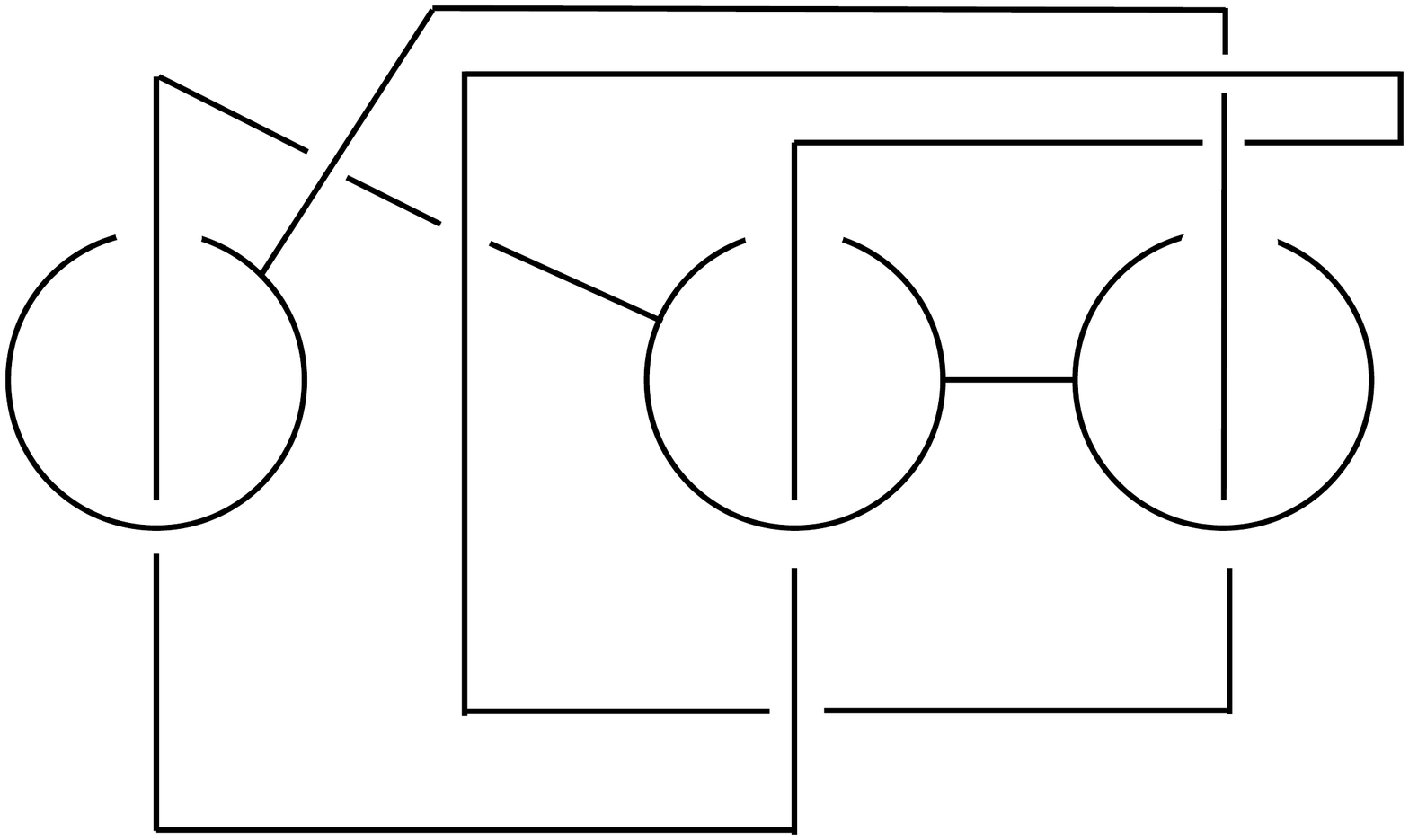}
\end{center}
\vspace{-13mm}
\caption{Diagram $D_{V^{'}}$ of the handlebody-knot $H_{V^{'}}$} 
\label{D_V^{'}}
\end{figure}

\begin{figure}[H]
\begin{center}
\includegraphics[width=15cm, height=10cm]{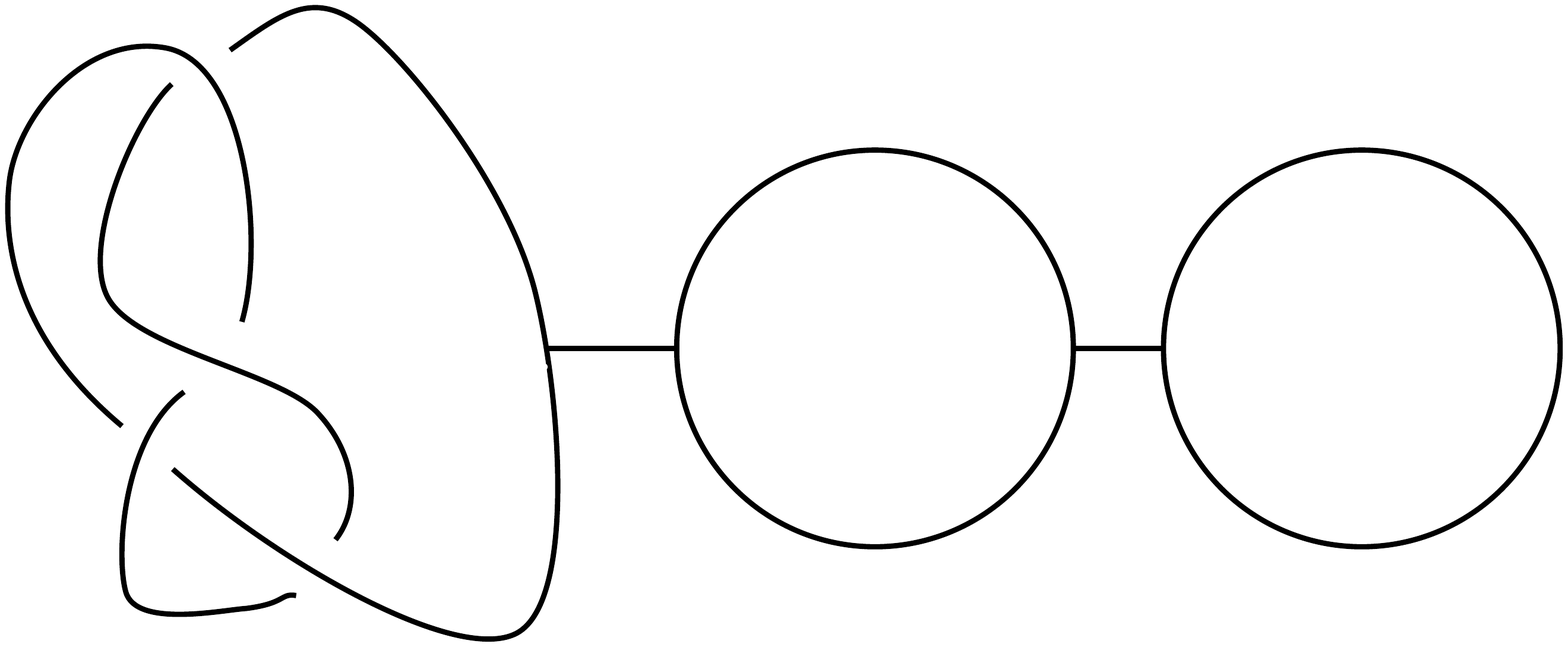}
\end{center}
\vspace{-13mm}
\caption{Diagram $D_{W^{'}}$ of the handlebody-knot $H_{W^{'}}$}
\label{D_W^{'}}
\end{figure}

\section{Further results}
In this section, we study the relationship between surfaces in $S^3$ and 2-component handlebody-links.
We also study closed connected orientable surfaces in $S^3$ from geometric viewpoints.
Then, we construct a geometric invariant of surfaces in $S^3$.
We also give a necessary condition of the existence of a surface which is obtained from a $2$-component handlebody-link corresponding to minimal genus Heegaard splittings of the closures of the connected components of the exterior of the surface.
We denote the interval $[-1,1]$ and the unit 2-disk by $D^1$ and $D^2$, respectively.

\subsection{Surfaces in $S^3$ and handlebody-links}

Let us start with a classification of embedded surfaces into three classes.
We denote by $V_F$ and $W_F$ the closures of the connected components of the exterior of a given embedded surface $F$.
\begin{defi}\label{surfaces}
Let $F$ be a closed connected orientable surface in $S^3$.
\begin{itemize}
\item [(i)] The surface $F$ is said to be an {\it unknotted surface} if both $V_{F}$ and $W_F$ are homeomorphic to handlebodies.
\item [(ii)] The surface $F$ is said to be a {\it knotted surface} if exactly one of $V_{F}$ or $W_F$ is homeomorphic to a handlebody.
\item [(iii)] The surface $F$ is said to be a {\it bi-knotted surface} if neither $V_{F}$ nor $W_F$ is homeomorphic to a handlebody.
\end{itemize}
\end{defi}

We first give a necessary and sufficient condition of the existence of a closed connected orientable surface in $S^3$ corresponding to a given 2-component handlebody-link.
We use Kneser's theorem and a characterization of $3$-manifolds (with connected boundary) whose fundamental groups are free.
We also use Lemma \ref{irreducible}, then we have the following.

\begin{prop}\label{iff}
Let $L=H_{1}\sqcup H_{2}$ be a $2$-component handlebody-link.
We denote by $E(L):=S^{3}\setminus{\rm int}(L)$ the exterior of $L$.
Then, there exists a closed connected orientable surface $F$ in $S^3$ such that $\partial H_{1}$ and $\partial H_{2}$ are Heegaard surfaces of the connected components of the exterior of $F$ if and only if the fundamental group $\pi_{1}(E(L))$ is given by free products of the fundamental group of a closed connected orientable surface and some infinite cyclic groups.
\end{prop}

Generally, since a Heegaard splitting of a compact connected orientable $3$-manifold is not unique, then a 2-component handlebody-link associated to a closed connected orientable surface in $S^3$ is not unique.
However, by using the Reidemeister--Singer theorem, we have the following.

\begin{prop}
Let $F$ be a closed connected orientable surface in $S^3$ and $L$ be an associated $2$-component handlebody-link of $F$, respectively.
Then, $L$ is unique up to stabilizations of Heegaard splittings of the closures of the connected components of  $S^{3}\setminus F$.
\end{prop}

\begin{prop}\label{injectibity}
Let $F_1$ and $F_2$ be closed connected orientable surfaces in $S^3$.
We denote by $V_i$ and $W_i$ the closures of connected components of $S^{3}\setminus F_{i}$ $(i=1,2)$.
Let $V_{i}=H_{V_{i}}\cup C_{V_{i}}$ and $W_{i}=H_{W_{i}}\cup C_{W_{i}}$ be Heegaard splittings of $V_i$ and $W_i$ $(i=1, 2)$, respectively, where $H_{V_{i}}$ and $H_{W_{i}}$ are handlebodies and $C_{V_{i}}$ and $C_{W_{i}}$ are compression bodies.
Let $L_{1}=H_{V_{1}}\sqcup H_{W_{1}}$ and $L_{2}=H_{V_{2}}\sqcup H_{W_{2}}$ be associated $2$-component handlebody-links of $F_{1}$ and $F_2$, respectively.
If $L_{1}$ and $L_{2}$ are stably equivalent, then $F_{1}$ and $F_{2}$ are isotopic.
\end{prop}

\begin{proof}
We assume that $H_{V_{1}}$ is mapped to $H_{V_{2}}$ and $H_{W_{1}}$ is mapped to $H_{W_{2}}$ by an isotopy of $S^3$.
We denote by $E(L_{1})$ the exterior of $L_{1}$.
Then, by Lemma \ref{irreducible}, we can show that $E(L_{1})$ is irreducible.
If both $C_{V_{1}}$ and $C_{W_{1}}$ admit $1$-handles, that is, $C_{V_{1}}$ and $C_{W_{1}}$ are constructed from the product manifolds $F_{1} \times [0,1]$ by attaching $1$-handles, then we can show that $\partial H_{V_{1}}$ and $\partial H_{W_{1}}$ are compressible in $E(L_{1})$.
Then, we can assume that $C_{V_{1}}$ and $C_{W_{1}}$ are incompressible neighborhoods(see Definition \ref{incomp}) of $\partial H_{V_{1}}$ and $\partial H_{W_{1}}$.
By Theorem \ref{mac}, $C_{V_{1}}$ and $C_{W_{1}}$ are unique up to ambient isotopy of $E(L_{1})$.
Therefore, by the use of isotopy extension theorem, we can show that $F_{1}$ and $F_{2}$ are isotopic.
If exactly one of $C_{V_{1}}$ or $C_{W_{1}}$ admits $1$-handles, let us suppose that $C_{V_{1}}$ admits $1$-handles.
Then, by using similar argument, we can show that $C_{V_{1}}$ is an incompressible neighborhood of $\partial H_{V_{1}}$ and unique.
On the other hand, since $C_{W_{1}}$ does not admit any $1$-handle, then by the triviality of a Heegaard splitting of a handlebody, $C_{W_{1}}$ is in the form $\partial H_{W_{1}}\times[0, 1]$.
Then we have the equivalence of $F_1$ and $F_2$.
If neither  $C_{V_{1}}$ nor $C_{W_{1}}$ admits any $1$-handle, then by using Waldhausen's theorem (Theorem \ref{walthm}), we can show that $F_{1}$ and $F_{2}$ are isotopic.
\end{proof}

Proposition \ref{injectibity} implies that an isotopy class of a surface in $S^3$ is uniquely determined from the stably equivalence class of an associated $2$-component handlebody-link of the surface.

\subsection{Results from geometric viewpoints}
We study surfaces in $S^3$ from geometric viewpoints, and construct a geometric invariant of surfaces.
\begin{defi}[Handle attaching \cite{suz2}]\label{handleattaching}
Let $F$ be a closed connected orientable surface embedded in $S^3$.
Let $h_{i}:D^{1}\times D^{2}\to S^3$ ($i=1, 2,  \ldots, n$) be embeddings of the $1$-handle such that $h_{i}(D^{1}\times D^{2})\cap F=h_{i}(\partial D^{1}\times D^{2})$ and $h_{i}(D^{1}\times D^{2})\cap h_{j}(D^{1}\times D^{2})=\emptyset$ (whenever $i\neq j$).
We call $h_{i}(D^{1}\times\partial D^{2})$ a {\it handle}.
We denote by $F(h_{1}, h_{2} \ldots, h_{n}):=F\cup(\sqcup_{i=1}^{n}h_{i}(D^{1}\times\partial D^{2}))\setminus\sqcup_{1=1}^{n}h_{i}(\partial D^{1}\times\text{int}(D^{2}))$ the surface obtained from $F$ by attaching $n$ handles.
Note that $F(h_{1}, h_{2} \ldots, h_{n})$ is again a closed connected orientable surface embedded in $S^3$ and that its genus is equal to $g(F)+n$.
\end{defi}

Closed connected orientable surfaces $F_{1}, F_{2}, \ldots F_{n}$ in $S^3$ are said to be {\it separated} if there exists mutually disjoint 3-balls $B^{3}_{i}$ such that $F_{i}\subset {\rm int}(B^{3}_{i})$ for each $i=1, 2, \ldots, n$.
If two closed connected orientable surfaces are separated by a $3$-ball $B$ and $S^{3} \setminus \text{int}(N(B))$, we call $B$ the \emph{associated ball}.

Given two separated embedded surfaces in $S^3$, we can construct a new embedded surface by using the following way, which is called the \emph{isotopy sum}.

\begin{defi}[Isotopy sum \cite{tsu1}]
Let $F_{1}$ and $F_{2}$ be closed connected orientable separated surfaces embedded in $S^3$ and $B^{3}$ be the associated ball as above, respectively.
Let $\varphi:D^{1}\times D^{2}\to S^{3}$ be an embedding of the 1-handle such that $\varphi(D^{1}\times D^{2})\cap F_{1}=\varphi(\{-1\}\times D^{2})$, $\varphi(D^{1}\times D^{2})\cap F_{2}=\varphi(\{1\}\times D^{2})$, and $\varphi(D^{1}\times D^{2})\cap\partial B^{3}=\varphi(\{0\}\times D^{2})$.
We define the {\it isotopy sum} of $F_{1}$ and $F_2$, denoted by $F_{1}\sharp_{\text{iso}}F_{2}$, by $F_{1}\sharp_{\text{iso}}F_{2}:=F_{1}\cup F_{2}\cup\varphi(D^{1}\times\partial D^{2})\setminus \varphi(\partial D^{1}\times\text{int}(D^{2}))$.
\end{defi}

We note that the isotopy sum of two separated embedded surfaces does not depend on the choice of the associated ball and the choice of an embedding $\varphi$ up to an isotopy of $S^3$.
Moreover, it does not depend on the order of $F_1$ and $F_2$.
We also note that $(F_{1}\sharp_{\rm iso}F_{2})\sharp_{\rm iso}F_{3}\cong F_{1}\sharp_{\rm iso}(F_{2}\sharp_{\rm iso}F_{3})$ for separated closed connected orientable surfaces $F_{1}, F_{2}$ and $F_{3}$ in $S^3$.
In this sense the isotopy sum is well-defined.
For an embedded surface $F$, if $F$ is isotopic to the isotopy sum $F_1 \sharp_{\rm iso} F_2 \sharp_{\rm iso} \cdots \sharp_{\rm iso} F_k$ for separated surfaces $F_1, F_2, \ldots, F_k$, we call it a \emph{decomposition of $F$ into factors} $F_i$, $i = 1, 2,\ldots, k$.

In order to construct a geometric invariant of closed connected orientable surfaces in $S^3$,  we introduce the notion of the {\it tunnel number} of a compact connected orientable $3$-manifold with connected boundary embedded in $S^3$.
\begin{defi}[Tunnel number]\label{tunnel}
Let $V$ be a compact connected orientable $3$-manifold with connected boundary embedded in $S^3$.
The {\it tunnel number} of $V$ is the minimal number $n$ of mutually disjoint 1-handles such that 
$\text{cl}(S^{3}\setminus(V\bigcup\sqcup_{i=1}^{n}h_{i}(D^{1}\times D^{2})))$ is homeomorphic to a handlebody and $V\cap h_{i}(D^{1}\times D^{2})=h_{i}(\partial D^{1}\times D^{2})$, where $h_i$ $(i=1, 2, \ldots n)$ are embeddings of the $1$-handle $D^{1}\times D^{2}$ into $S^3$.
We denote by $t(V)$ the tunnel number of $V$.
\end{defi}

It is not hard to see that since every compact connected orientable 3-manifold possibly with boundary admits a Heegaard splitting, $t(V)$ is always finite.
Let $F$ be a closed connected orientable surface in $S^3$.
We denote by $V_{F}$ and $W_{F}$ the closures of the connected components of $S^{3}\setminus F$.
By Definition \ref{tunnel}, we also see that the Heegaard genus of $W_F$ is given by $g(F)+t(V_{F})$.
Similarly, the Heegaard genus of $V_F$ is given by $g(F)+t(W_{F})$.

We note that if $V_{F}$ is a handlebody-knot, then the tunnel number of $V_{F}$ coincides with that of a handlebody-knot \cite{ishii2}.

\begin{prop}\label{prop1}
Let $F$ be a closed connected orientable surface in $S^3$.
Let $h_{i}:D^{1}\times D^{2}\to S^3$ $(i=1, 2,  \ldots, n)$ be embeddings of the 1-handle such that $h_{i}(D^{1}\times D^{2})\cap F=h_{i}(\partial D^{1}\times D^{2})$ and $h_{i}(D^{1}\times D^{2})\cap h_{j}(D^{1}\times D^{2})=\emptyset$ (whenever $i\neq j$).
Then there exists a finite numbers of handles such that $F(h_{1}, h_{2} \ldots, h_{n})$ is an unknotted surface.
\end{prop}

\begin{proof}
Let $V_F$ and $W_F$ be the closures of the connected components of $S^{3}\setminus F$.
Then, we have $S^3=V_{F}\cup W_{F}$.
Let $t(V_{F})$ and $t(W_{F})$ be the tunnel numbers of $V_F$ and $W_F$.
Then, by the definition of the tunnel number, we obtain two handlebody-knots $H_{1}$ and $H_{2}$ such that $H_{1}\cup H_{2}=S^{3}$ and $\partial H_{1}=\partial H_{2}$ by attaching $t(V_{F})+t(W_{F})$ $1$-handles to the common boundary $F=V_{F}\cap W_{F}$ of $V_{F}$ and $W_F$.
Since the union of $H_1$ and $H_2$ along their boundaries gives a Heegaard splitting of $S^3$.
Then, combining Waldhausen's theorem(Theorem \ref{walthm}), the common boundary of $H_1$ and $H_2$ is an unknotted surface.
\end{proof}

We introduce the definition of the handle number of embedded surface.

\begin{defi}[Handle number]
Let $F$ be a closed connected orientable surface in $S^3$.
Let $h_{i}:D^{1}\times D^{2}\to S^3$ ($i=1, 2,  \ldots, n$) be embeddings of the 1-handle such that $h_{i}(D^{1}\times D^{2})\cap F=h_{i}(\partial D^{1}\times D^{2})$ and $h_{i}(D^{1}\times D^{2})\cap h_{j}(D^{1}\times D^{2})=\emptyset$ (whenever $i\neq j$).
Let $F(h_{1}, h_{2}, \ldots, h_{n})$ be a surface obtained from $F$ by attaching $n$ handles to $F$.
The {\it handle number} of $F$, denoted by $h(F)$, is the minimal number $n$ of handles such that $F(h_{1}, h_{2}, \ldots, h_{n})$ is an unknotted surface.
\end{defi}

Let $F$ be a closed connected orientable surface in $S^3$.
We denote by $V_F$ and $W_F$ the closures of the connected components of $S^{3}\setminus F$.
We show the relationship between the handle number of $F$ and the tunnel numbers of $V_{F}$ and $W_{F}$.

\begin{prop}\label{theo1}
Let $F$ be a closed connected orientable surface in $S^3$.
We denote by $V_F$ and $W_F$ the closures of the connected components of $S^{3}\setminus F$.
Let $t(V_{F})$ and $t(W_{F})$ be the tunnel numbers of $V_{F}$ and $W_{F}$.
Then, for the handle number of $F$,  we have $h(F)=t(V_{F})+t(W_{F})$.
\end{prop}

\begin{proof}
Let $n$ be the handle number of $F$.
Let $h_{i}:D^{1}\times D^{2}\to S^{3}\ (i=1, 2, \ldots, n)$ and $\{h_{1}, h_{2}, \ldots, h_{n}\}$ be embeddings of the $1$-handle satisfying the conditions given in Definition \ref{handleattaching} and the set of attached handles to $F$ such that $F(h_{1}, h_{2}, \ldots, h_{n})$ is an unknotted surface, respectively.
We divide the set of handles $\{h_{1}, h_{2}, \ldots, h_{n}\}$ into the disjoint union of the sets of handles $\{h_{1}^{V}, h_{2}^{V}, \ldots, h_{k}^{V}\}$ and $\{h_{1}^{W}, h_{2}^{W, }\ldots, h_{l}^{W}\}$, where $h_{i}^{V}(D^{1}\times D^{2})$ is attached to $V_{F}$, and $h_{j}^{W}(D^{1}\times D^{2})$ is attached to $W_F$.
Then, by Proposition \ref{prop1}, we have $n=k+l\leq t(V_{F})+t(W_{F})$.
On the other hand, by the definition of the tunnel number, we have $t(V_{F})\leq k$ and $t(W_{F})\leq l$.
Then, we obtain $t(V_{F})+t(W_{F})\leq k+l$.
Therefore, we have $h(F)=k+l=t(V_{F})+t(W_{F})$.
\end{proof}

As an immediate sequence, we have that the unordered pair $(t(V_{F}),t(W_{F}))$ of non-negative integers is an invariant of $F$.
Furthermore, it is not hard to see the following.
Let $V_{F}=H_{V}\cup C_{V}$ and $W_{F}=H_{W}\cup C_{W}$ be minimal genus Heegaard splittings of $V_{F}$ and $W_{F}$ consisting of pairs of a handlebody $H_{V}$ and a compression body $C_{V}$ and a handlebody $H_{W}$ and a compression body $C_{W}$, respectively.
If $F$ is an unknotted surface, then both $H_{V}$ and $H_{W}$ are trivial handlebody-knots.
Hence $t(V_{F})=0$ and $t(W_{F})=0$.
If $F$ is a knotted surface, then exactly one of $V_{F}$ or $W_{F}$ is homeomorphic to a handlebody, and the other is not homeomorphic to a handlebody.
Hence exactly one of $t(V_{F})$ or $t(W_{F})$ is equal to zero, and the other is a positive integer.
If $F$ is a bi-knotted surface, then 
we have that neither $V_{F}$ nor $W_{F}$ is homeomorphic to a handlebody.
Hence both $t(V_{F})$ and $t(W_{F})$ are positive integers.

It is not hard to see that since the unordered pair $(t(V_{F}),t(W_{F}))$ of non-negative integers is an invariant of $F$,  then the handle number $h(F)$ is also an invariant of $F$.
We prove the following corollary.

\begin{coro}\label{cor1}
Let $F$ be a closed connected orientable surface in $S^3$.
We denote by $V_{F}$ and $W_{F}$ the closures of the connected components of $S^{3}\setminus F$.
Let $n$ be the handle number of $F$.
Let $h_{i}:D^{1}\times D^{2}\to S^{3}\ (i=1, 2, \ldots, n)$ be embeddings of the $1$-handle satisfying the conditions given in Definition \ref{handleattaching}.
Let $\{h_{1}, h_{2}, \ldots, h_{n}\}=\{h_{1}^{V}, h_{2}^{V}, \ldots, h_{k}^{V}\}\sqcup\{h_{1}^{W}, h_{2}^{W}, \ldots, h_{l}^{W}\}$ be the set of attached handles, where $h_{i}^{V}(D^{1}\times D^{2})$ is attached to $V_{F}$, and $h_{j}^{W}(D^{1}\times D^{2})$ is attached to $W_F$.
Let $t(V_{F})$ and $W_{F}$ be the tunnel numbers of $V_F$ and $W_F$.
Then, $k=t(V_{F})$ and $l=t(W_{F})$.
\end{coro}

\begin{proof}
In the proof of Theorem \ref{theo1}, we saw that $k+l=t(V_{F})+t(W_{F})$, $t(V_{F})\leq k$, and $t(W_{F})\leq l$.
Let us assume that $t(V_{F})<k$.
Then, we have $t(W_{F})>l$.
This is a contradiction.
Hence $k=t(V_{F})$.
The same holds for $l$, that is, $l=t(W_{F})$.
\end{proof}

\begin{prop}\label{constant}
Let $F$ be a closed connected orientable surface in $S^3$.
We denote by $V_F$ and $W_F$ the closures of the connected components of $S^{3}\setminus F$.
For any minimal genus Heegaard splitting of $V_{F}$, say $V_{F}=H_{V}\cup C_{V}$, the tunnel number $t(H_{V})$ of the handlebody-knot $H_V$ does not depend on the choice of minimal genus Heegaard splittings.
The same holds for any minimal genus Heegaard splitting of $W_F$.
\end{prop}

\begin{proof}
Let $V_{F}=H_{V}\cup C_{V}$ and $V_{F}=H^{'}_{V}\cup C^{'}_{V}$ be minimal genus Heegaard splittings of $V_{F}$.
Let $n=k+l$ be the handle number of the surface $F$, where $k$ and $l$ are the numbers of 1-handles attached to $V_{F}$ and $W_F$.
Let $h_{i}:D^{1}\times D^{2}\to S^{3}\ (i=1, 2, \ldots, n)$ and $\{h_{1}, h_{2}, \ldots, h_{n}\}$ be embeddings of the $1$-handle satisfying the conditions given in Definition \ref{handleattaching} and the set of attached handles to $F$ such that $F(h_{1}, h_{2}, \ldots, h_{n})$ is an unknotted surface, respectively.
Then, we divide the set of handles $\{h_{1}, h_{2}, \ldots, h_{n}\}$ into the disjoint union of two sets $\{h_{1}^{V}, h_{2}^{V}, \ldots, h^{V}_{k}\}$ and $\{h_{1}^{W}, h_{2}^{W}, \ldots, h^{W}_{l}\}$, where $h^{V}_{i}$ is a handle attached to $V_{F}$, and $h^{W}_{j}$ is a handle attached on $W_{F}$.
By the definition of the tunnel number of $W_F$, ${\rm cl}(S^{3}\setminus (W_{F}\bigcup \sqcup_{i=1}^{t(W_{F})}h_{i}(D^{1}\times D_{2})))$ is isotopic to a handlebody.
In particular, since $V_{F}=H_{V}\cup C_{V}$ is a minimal genus Heegaard splitting, we can assume that ${\rm cl}(S^{3}\setminus (W_{F}\bigcup \sqcup_{i=1}^{t(W_{F})}h_{i}(D^{1}\times D^{2})))$ is homeomorphic to $H_{V}$.
Hence, $\partial H_{V}$ is obtained from $F$ by attaching $t(W_{F})$ handles.
Then, by the definition of the handle number of $F$, we can also assume that we obtain an unknotted surface from $\partial H_{V}$ by attaching $k$ handles.
Moreover, by the definition of the tunnel number of $H_V$, we obtain an unknotted surface from $H_V$.
Then,  we have $k=t(H_{V})$.
By using a similar argument, we also have $k=t(H_{V}^{'})$.
The same holds for minimal genus Heegaard splittings of $W_F$.
\end{proof}

Let $F$ be a closed connected orientable surface embedded in $S^3$.
We denote by $V_{F}$ and $W_{F}$ the closures of the connected components of $S^{3}\setminus F$.
By considering Heegaard splittings $V_{F}$ and $W_{F}$, we obtain an associated 2-component handlebody-link $L$ of $F$.
By applying the Reidemeister--Singer theorem, we showed that $L$ is uniquely determined from $F$ up to stabilizations of Heegaard splittings of $V_{F}$ and $W_{F}$.
Conversely, let us consider a 2-component handlebody-link $L^{'}$.
Then, does there always exist a closed connected orientable surface $F^{'}$ embedded in $S^{3}$ such that the connected components of $\partial L^{'}$ are Heegaard surfaces of $V_{F^{'}}$ and $W_{F^{'}}$?
A necessary and sufficient condition was given in Proposition \ref{iff}.
We consider such a problem from the viewpoints of the handle number and the tunnel number.

By Proposition \ref{constant}, we have the following corollary.

\begin{coro}
Let $L=H_{1}\sqcup H_{2}$ be a $2$-component handlebody-link.
Let us assume that $t(H_{1})-t(H_{2})\neq g(\partial H_{2})-g(\partial H_{1})$.
Then, there does not exist a closed connected orientable surface in $S^3$ such that $L$ is obtained from minimal Heegaard splittings of the closures of the exterior of the surface.
\end{coro}
Let $M$ be a compact connected orientable 3-manifold whose boundary is not homeomorphic to a $2$-sphere.
Then, $M$ is said to be {\it $\partial$-irreducible} if for every properly embedded 2-disk $D$ in $M$, $\partial D$ bounds a 2-disk on $\partial M$.
We assume that the boundary of $M$ is connected, and that $M$ is not homeomorphic to a $3$-ball.
Then, $M$ is said to be {\it $\partial$-prime} if for any decomposition $M_{1}\natural M_{2}$, where $M_1$ and $M_2$ are compact connected orientable $3$-manifolds, either $M_1$ or $M_2$ is homeomorphic to a 3-ball, where $M_{1}\natural M_{2}$ is the disk sum, or the boundary connected sum,  of $M_1$ and $M_2$.
Then, the following is known.

\begin{prop}[{\rm \cite{suz1}}\label{propsuzu1}]
Let $M$ be a compact connected orientable $3$-manifold embeddable in $S^3$ with connected boundary with $g(\partial M)\geq2$.
Then, the following conditions are equivalent.
\begin{itemize}
\item[\rm (i)] The $3$-manifold $M$ is $\partial$-prime.
\item[\rm (ii)] The $3$-manifold $M$ is $\partial$-irreducible.
\item[\rm (iii)] The fundamental group $\pi_{1}(M)$ of $M$ is indecomposable with respect to free products.
\end{itemize}
\end{prop}

We now present the notion of prime surfaces in $S^3$.
\begin{defi}[Prime surface \cite{tsu1}]
Let $F$ be a closed connected orientable surface in $S^3$ with $g(F)\geq1$.
Then, $F$ is said to be {\it prime} if for any decomposition $F\cong F_{1}\sharp_{\text{iso}}F_{2}$, either $F_{1}$ or $F_{2}$ is a surface of genus 0.
We note that every closed connected orientable surface of genus 1 is prime.
\end{defi}

About closed connected orientable prime surfaces in $S^3$, the following is known.

\begin{theo}[{\rm \cite{tsukui}}\label{theotsukui}]
Let $F$ be a closed connected orientable surface of genus $2$ in $S^3$.
Then, $F$ is prime if and only if either $\pi_{1}(V_{F})$ or $\pi_{1}(W_{F})$ is indecomposable with respect to free products.
\end{theo}

Let $M$ be a compact connected orientable 3-manifold with connected boundary embedded in $S^3$.
Then, $M$ is said to be {\it reducibly embedded} if there exists a 2-sphere $S$ in $S^3$ such that $M\cap S$ is a properly embedded 2-disk $D$ in $M$, and that $D$ separates $M$ into two parts which are not homeomorphic  to 3-balls.
Otherwise, $M$ is said to be {\it irreducibly embedded}.
We call a reducibly embedded handlebody-knot a {\it reducible handlebody-knot}.
Similarly, we call an irreducibly embedded handlebody-knot an {\it irreducible handlebody-knot}.

Let us now focus on prime surfaces of genus $2$ embedded in $S^3$.
Let $F$ be a prime surface of genus $2$ in $S^3$.
We denote by $V_F$ and $W_F$ the closures of the connected components of $S^{3}\setminus F$.
By virtue of Theorem \ref{theotsukui} and Proposition \ref{propsuzu1}, we know that exactly one of $V_F$ or $W_F$ is $\partial$-prime.
On that assumption, we have the following proposition.

\begin{prop}\label{theo2}
Let $F$ be a closed connected orientable prime surface of genus $2$ in $S^3$.
We denote by $V_F$ and $W_F$ the closures of the connected components of $S^{3}\setminus F$.
We assume that exactly one of $V_{F}$ or $W_F$ is $\partial$-prime.
Let $V_{F}=H_{V}\cup C_{V}$ and $W_{F}=H_{W}\cup C_{W}$ be minimal genus Heegaard splittings of $V_{F}$ and $W_F$, respectively.
Then, exactly one of $H_V$ or $H_W$ is a reducible handlebody-knot, and the other is an irreducible handlebody-knot.
\end{prop} 

\begin{proof}
By the assumption, we assume that $V_{F}$ is not $\partial$-prime and $W_{F}$ is $\partial$-prime.
Since $V_{F}$ is not $\partial$-prime, there exists a decomposition $V_{F}=V_{F}^{1}\natural V_{F}^{2}$ such that neither $V_{F}^{1}$ nor $V_{F}^{2}$ is homeomorphic to a 3-ball $B^{3}$.
Then, by the use of the solid torus theorem, $V_{F}^{i}$ $(i=1, 2)$ are either a handlebody-knot of genus 1 or the exterior of a hanlebody-knot of genus $1$.

\noindent
(i) In case of both $V_{F}^{1}$ and $V_{F}^{2}$ are handlebody-knots of genus 1, since $F$ is a prime surface, then $V_{F}=V_{F}^{1}\natural V_{F}^{2}$ is an irreducible handlebody-knot.
Since the Heegaard splittings of a handlebody-knot are standard, then $H_{V}$ is an irreducible handlebody-knot.
On the other hand, $W_{F}$ is the exterior of the handlebody-knot $V_{F}$.
Then the Heegaard surfaces of $W_{F}$ are unknotted surfaces.
Therefore $H_{W}$ is a reducible handlebody-knot.

\noindent
(ii) In case of $V_{F}^{1}$ is a handlebody-knot of genus 1, and $V_{F}^{2}$ is the exterior of a non-trivial handlebody-knot of genus 1, let us assume that $V_{F}^{1}$ is obtained from a $3$-ball by attaching a $1$-handle, and $V_{F}^{2}$ is obtained from a 3-ball by removing a 2-handle. 
Then $V_{F}$ is constructed from $V_{F}^{1}$ and $V_{F}^{2}$ by connecting them by a $1$-handle $h$ so as to $h$ throughs the removed 2-handle from the $3$-ball at least one time up to isotopy.
Since $V_{F}^{1}$ is a handlebody-knot, then we have a minimal genus Heegaard splitting of $V_F$ by removing 2-handles from $V_{F}^{2}$.
Therefore by considering a minimal genus Heegaard splitting of $V_{F}^{2}$, we obtain the trivial handlebody-knot of some genus $\geq 2$ from $V_{F}^{2}$.
Such the trivial handlebody-knot is obtained from $V_{F}^{2}$ by removing 2-handles.
Then by applying the deformation depicted in Figure \ref{Deformation} on $V_{F}^{2}$ several times after removing $2$-handles, we have the trivial handlebody-knot of some genus.
Moreover, by the assumption of the 1-handle, the 1-handle throughs every removed 2-handle from $V_{F}^{2}$ (Figure \ref{triv}).
Hence we see that $H_{V}$ is an irreducible handlebody-knot.
On the other hand, we have 
$W_{F}=\text{cl}(S^{3} \setminus V_{F})
=\text{cl}(S^{3} \setminus (V_{F}^{1} \cup V_{F}^{2} \cup h))
=\text{cl}(((B_{2}^{3} \cup \{ \text{1-handle} \}) \setminus (B_{1}^{3} \setminus \{\text{$2$-handle}\})) \setminus h)
= \text{cl}(((B_{2}^{3} \setminus B_{1}^{3}) \cup \{\text{$1$-handles}\})\setminus h)$, where $B_{1}^{3}$ and $B_{2}^{3}$ are $3$-balls such that $\text{\{$2$-handle\}} \subset B_{1}^{3} \subset \text{int}(B_{2}^{3}$).
We note that the $1$-handle $h$ is removed from $\text{cl}((B_{2}^{3} \setminus B_{1}^{3}) \cup \{ \text{1-handles} \})$ so as to it throughs both attached $1$-handles(Figure \ref{$W_{F}$}).
Finally, in order to obtain a Heegaard splitting of $W_F$, we remove $2$-handles from $B_{2}^{3}$ and deform the resulting $3$-manifold(Figures \ref{remove2-handle} and Figure \ref{$W_{F}$heegaard}).
In this way, after removing appropriate numbers of $2$-handles, we obtain a minimal genus Heegaard splitting of $W_F$.
By applying the above procedure, we can see that $H_W$ is a reducible handlebody-knot.
The same holds for in case of $V_{F}^{2}$ is a handlebody-knot of genus $1$ and $V_{F}^{1}$ is the exterior of a handlebody-knot of genus $1$.

\noindent
(iii) In case of both $V_{F}^{1}$ and $V_{F}^{2}$ are the exterior of handlebody-knots of genus 1, we apply the same methods to $V_{F}^{1}$ and $V_{F}^{2}$ introduced in (ii).
Then we can show that $H_V$ is an irreducible handlebody-knot and $H_W$ is a reducible handlebody-knot. 
\end{proof}

\begin{figure}[H]
\begin{center}
\includegraphics[width=15cm, height=10cm]{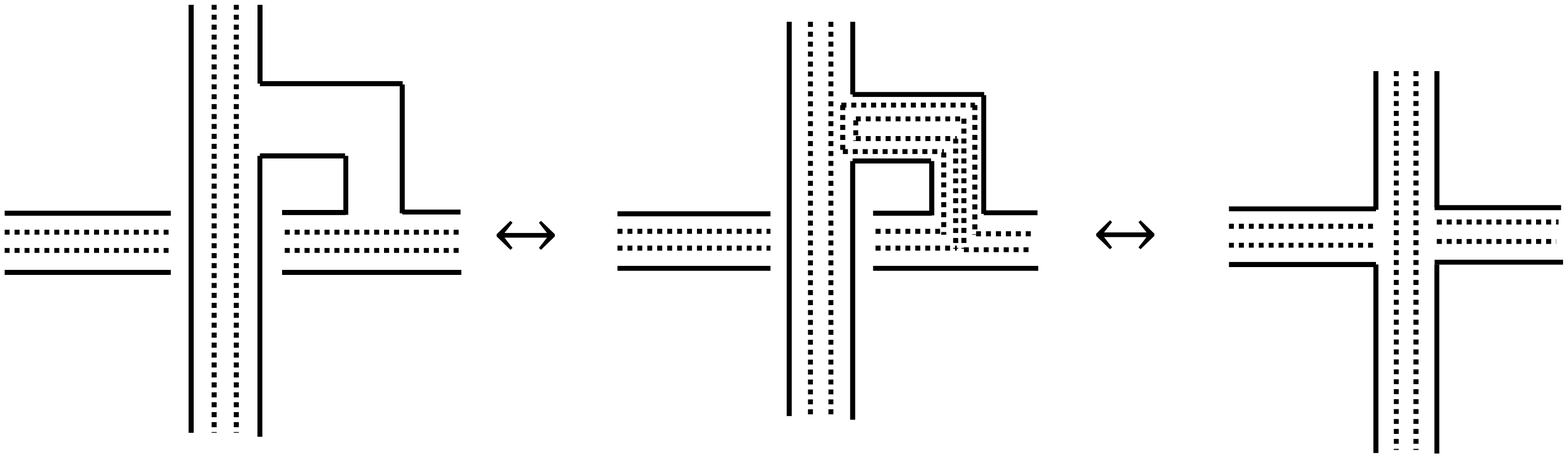}
\end{center}
\vspace{-13mm}
\caption{Deformation} 
\label{Deformation}
\end{figure}

\begin{figure}[H]
\begin{center}
\includegraphics[width=15cm, height=10cm]{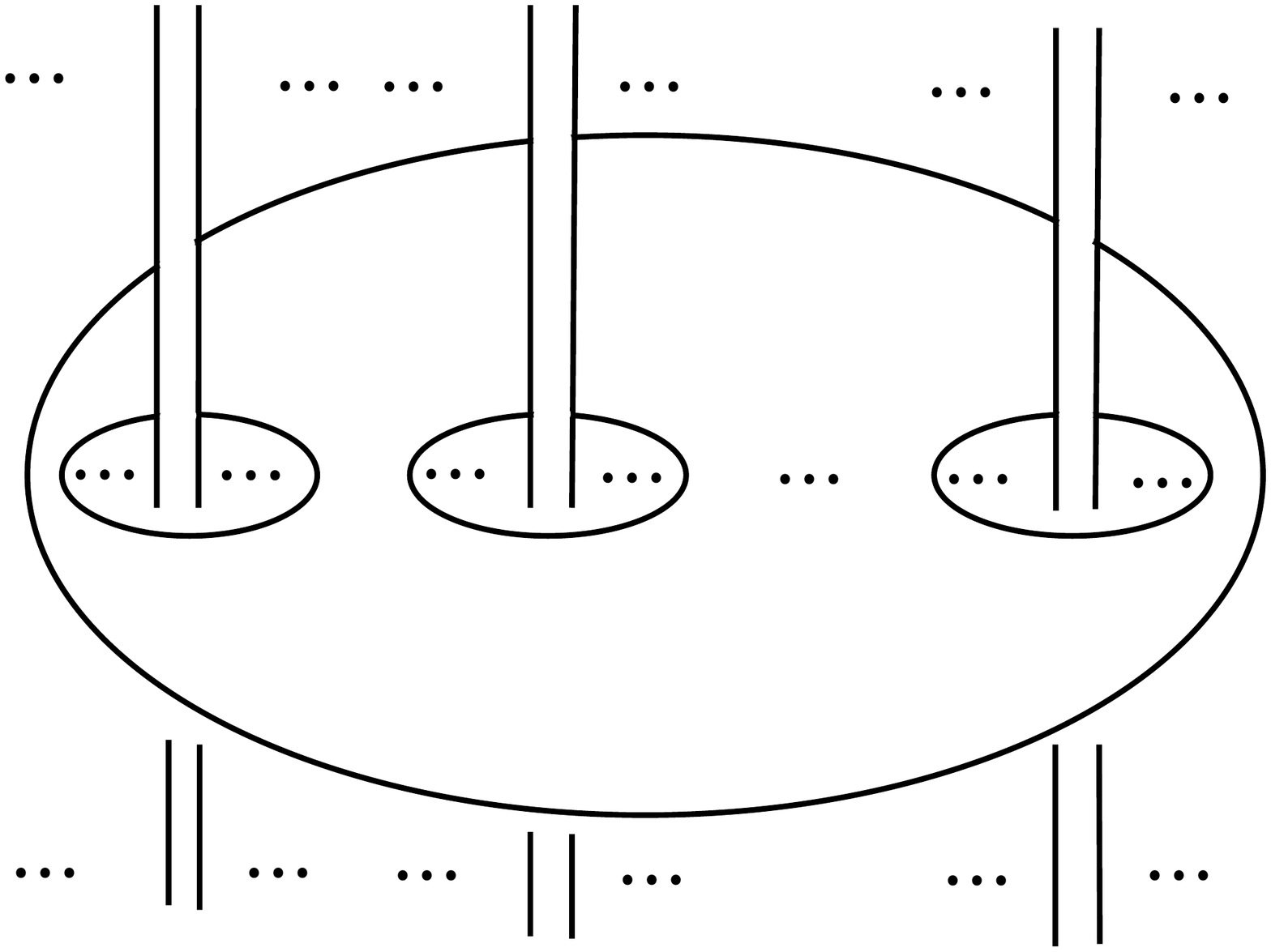}
\end{center}
\vspace{-13mm}
\caption{Part of the handlebody-knot $H_{V}$}
\label{triv}
\end{figure}

\begin{figure}[H]
\begin{center}
\includegraphics[width=15cm, height=10cm]{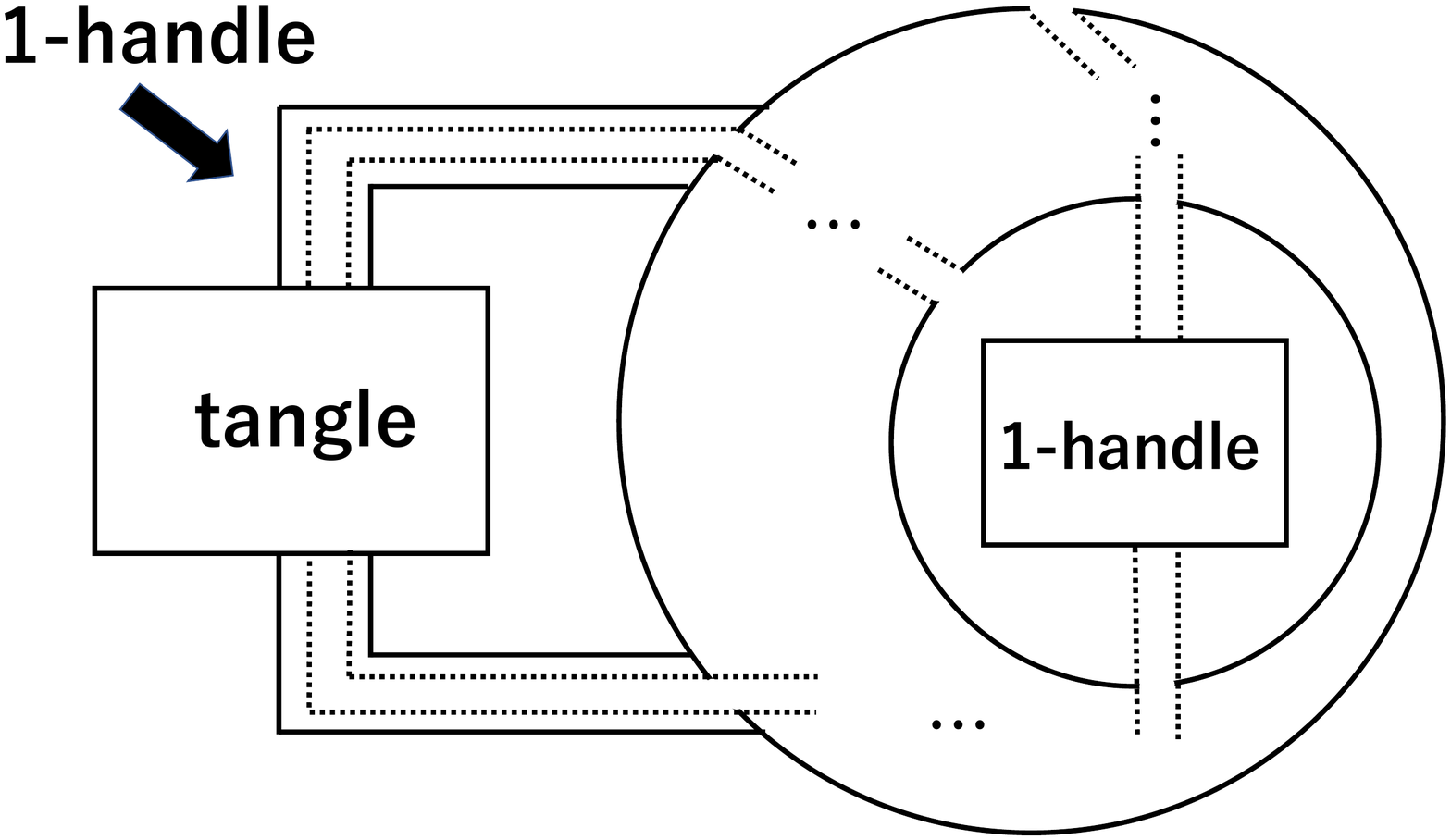}
\end{center}
\vspace{-13mm}
\caption{Exterior component $W_{F}$} 
\label{$W_{F}$}
\end{figure}

\begin{figure}[H]
\begin{center}
\includegraphics[width=15cm, height=10cm]{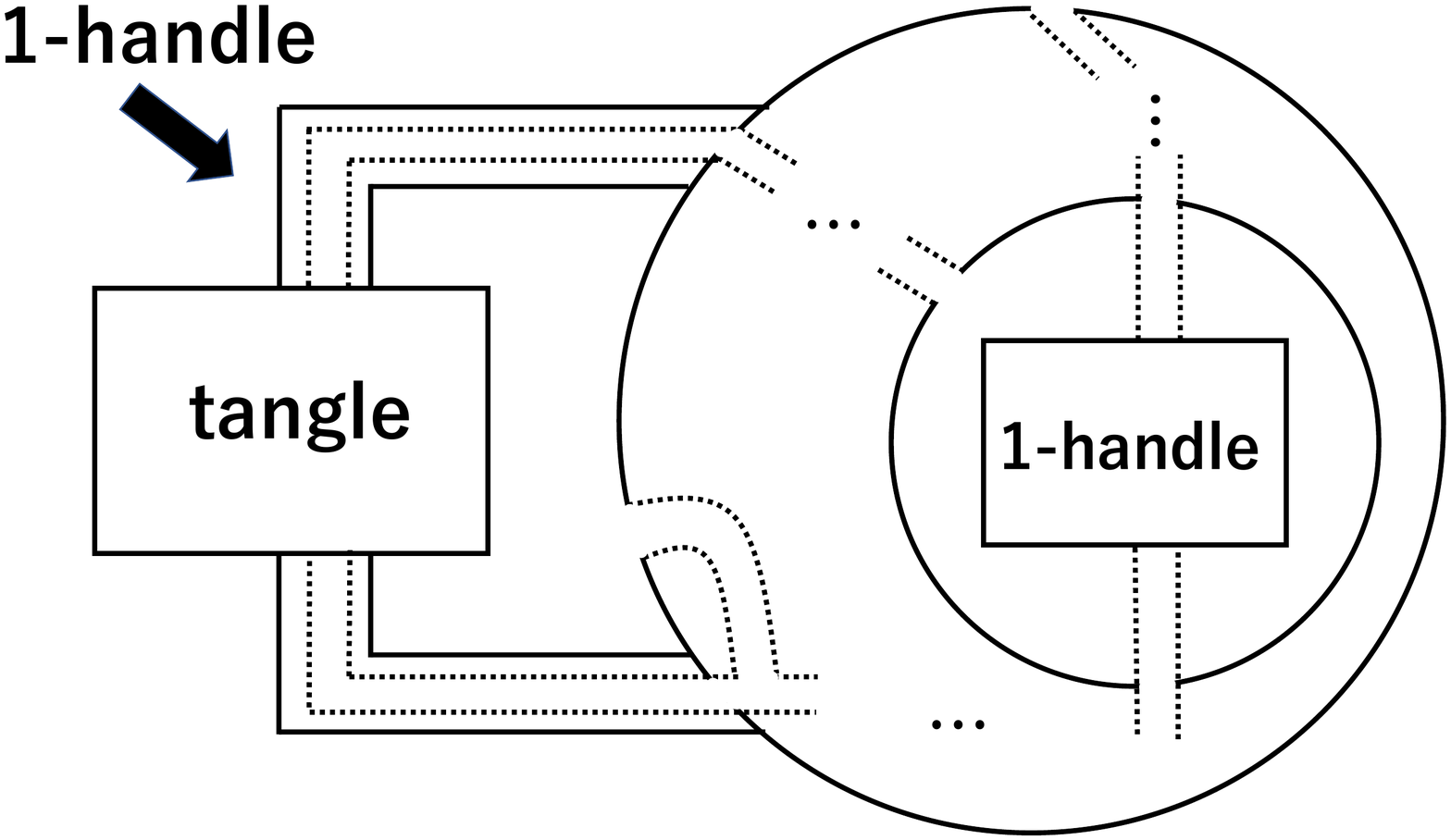}
\end{center}
\vspace{-13mm}
\caption{Removing of a 2-handle}
\label{remove2-handle}
\end{figure}

\begin{figure}[H]
\vspace{-5mm}
\begin{center}
\includegraphics[width=15cm, height=10cm]{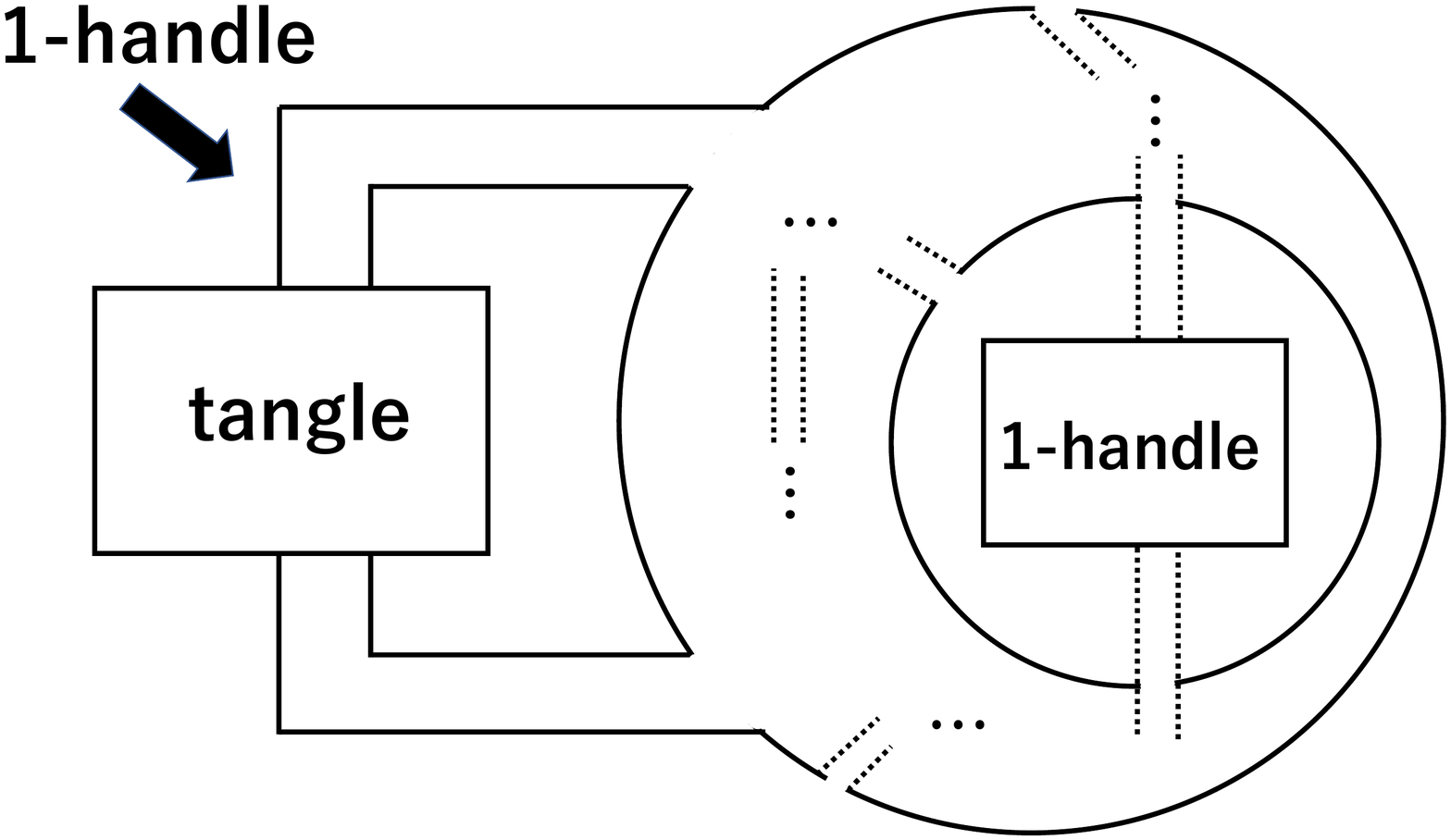}
\vspace{-10mm}
\caption{Procedure of a Heegaard splitting of $W_{F}$}
\label{$W_{F}$heegaard}
\end{center}
\end{figure}

\section*{Acknowledgement}
The author would like to express his deep gratitude to his supervisor, Professor Osamu Saeki, for his constant encouragement, many helpful comments, and discussions.
He would also like to thank Professor Masahico Saito for valuable comments and discussions.

Fujitsu Research, Kamiodanaka 4-1-1, Nakahara-ku, Kawasaki, Kanagawa, 211-8588, Japan.

{\it E-mail address}: h$\_$kurihara@fujitsu.com

\end{document}